\documentclass[12pt,twoside]{article}
\usepackage{fullpage}
\nonstopmode
\usepackage[latin1]{inputenc}
\usepackage[english]{babel}
\usepackage{clm-macros}
\usepackage{amsmath}
\usepackage{psfrag}
\usepackage{epsfig}
\usepackage{graphicx}
\usepackage{verbatim}
\usepackage{tikz}
\usepackage{enumerate}
\usepackage{datetime}

\usetikzlibrary{decorations}
\usetikzlibrary{decorations.pathmorphing}
\usetikzlibrary{arrows}
\tikzstyle{every node}=[circle, draw, fill=white,inner sep=0pt, minimum width=4pt]
\tikzstyle{nodelabel}=[rounded corners,fill=none,inner sep=5pt,draw=none]
\tikzstyle{matching} = [ultra thick]
\tikzset{snake/.style={decorate, decoration=snake}}

\newdateformat{UKvardate}{
\THEDAY\ \monthname[\THEMONTH], \THEYEAR}

\newcommand{\NB}{near-bipartite}
\newcommand{\PO}{pfaffian orientation}
\newcommand{\pema}{perfect matching}
\newcommand{\mc}{matching covered}
\newcommand{\mcg}{matching covered graph}
\newcommand{\bisub}{\mbox{bi-subdivision}}
\newcommand{\comi}{conformal minor}
\newcommand{\mami}{matching minor}
\newcommand{\Smi}{\mbox{$S$-minor}}

\newcommand{\pfaff}{pfaffian}
\newcommand{\nonpfaff}{\mbox{non-pfaffian}}
\newcommand{\JF}{\mbox{$J$-free}}
\newcommand{\JB}{\mbox{$J$-based}}
\newcommand{\PF}{\mbox{$\overline{C_6}$-free}}
\newcommand{\PB}{\mbox{$\overline{C_6}$-based}}
\newcommand{\KFF}{\mbox{$K_4$-free}}
\newcommand{\KFB}{\mbox{$K_4$-based}}
\newcommand{\KTTF}{\mbox{$K_{3,3}$-free}}
\newcommand{\KTTB}{\mbox{$K_{3,3}$-based}}
\newcommand{\TB}{\mbox{$\Theta$-based}}

\newcommand{\ed}{ear decomposition}

\newcommand{\KFdec}{\mbox{$K_4$-decoration}}

\newcommand{\claw}{$3$-claw}
\newcommand{\breaksta}{\hfill ~\linebreak }

\begin{document}

\title{Cubic graphs, {\Smi}s and {\comi}s}
\author{Nishad Kothari\thanks{Most of this work was done when the author was a postdoctoral researcher at the
University of Campinas, Brazil, and was supported by FAPESP Brazil \mbox{(2018/04679-1)}.}\\
\small IIT Madras, India\\[-0.8ex]
\small\tt nishadkothari@gmail.com\\
\and
Orlando Lee\\
\small University of Campinas, Brazil\\[-0.8ex]
\small\tt lee@ic.unicamp.br\\
\and
Cl{\'a}udio L. Lucchesi\\
\small University of S{\~a}o Paulo, Brazil\\[-0.8ex]
\small\tt lucchesi@ime.usp.br\\
\and
C{\^a}ndida Nunes da Silva\\
\small Federal University of S{\~a}o Carlos, Brazil\\[-0.8ex]
\small\tt candida@ufscar.br
}
\UKvardate
\date{\today}
  \maketitle 
  \thispagestyle{empty}

\begin{abstract} 

It is well-known that any class of simple graphs, say~$\mathcal{G}$, that is characterized by
finitely many forbidden minors, also admits a characterization by finitely many forbidden
topological minors; furthermore, the list of forbidden topological minors may be derived
from the list of forbidden minors.

We prove a similar result in Matching Theory.
Our Main Theorem states that any
class of {\mcg}s, say~$\mathcal{G}$, that is characterized by finitely many forbidden
{\Smi}s that are cubic, also admits a characterization by finitely many forbidden {\comi}s
that are cubic as well; once again, the list of forbidden {\comi}s may be derived
from the list of {\Smi}s.

In order to establish the above, we first prove that every \mcg\ has one of two graphs
as a conformal minor --- either $K_4$, or the $\Theta$ graph
(i.e., two vertices joined by three edges).
(In fact, we need and prove a much stronger statement.)
This is reminescent of a theorem due to Lov{\'a}sz: every nonbipartite \mcg\ has
one of two graphs as a conformal minor --- either $K_4$, or the triangular prism~$\overline{C_6}$.

As applications of our Main Theorem, we deduce known `forbidden \comi\ characterizations' of
\pfaff\ near-bipartite graphs, and of \pfaff\ solid graphs, using their respective
known `forbidden \Smi\ characterizations'.

\end{abstract}


\bigskip
For general  graph-theoretic notation  and terminology, we  follow the
book by Bondy and Murty~\cite{bomu08}.   All graphs considered in this
paper are loopless. Unless otherwise  stated, we allow multiple (i.e.,
parallel) edges.

\section{Main Story}

A connected graph, of order two or more, is {\em \mc}
if each edge lies in some perfect matching. For several problems
that pertain to the perfect matchings of a graph (such as counting
the number of perfect matchings, or the related problem
of deciding whether a graph admits a \pfaff\ orientation),
one may restrict attention to {\mcg}s without
entailing any loss of generality.

\smallskip
A vertex of a graph is {\it cubic} if it has degree precisely three.
A graph is {\em subcubic}
if each of its vertices has degree at most three,
and it is {\it cubic} if each of its vertices is cubic.
(Note that a cubic graph is also subcubic.)
Using Tutte's $1$-factor theorem, one may deduce that a cubic graph
is \mc\ if and only if it is $2$-connected. The smallest $2$-connected cubic
graph is the $\Theta$ graph (i.e., two vertices joined by three edges);
whereas the smallest nonbipartite $2$-connected cubic graphs
are $K_4$ and the triangular prism~$\overline{C_6}$.

\smallskip
Graph containment notions of {\em minors}
and {\em topological minors} play crucial roles in graph theory.
The latter is a stricter notion than the former --- in the sense that ---
if a graph $H$ is a topological minor of a graph~$G$
then $H$ is also a minor of~$G$; however, the converse is not true in general.
That being said, the two notions coincide whenever the ``smaller graph''
is subcubic. That is, a subcubic graph $H$ is a minor of a graph~$G$
if and only if $H$ is a topological minor of~$G$.
The following is a well-known theorem (see \cite[pg. 352, ex. 33]{dies05}).

\begin{thm}
\label{thm:minors-vs-topological-minors}
Any class of simple
graphs, say~$\mathcal{G}$,
that is characterized by finitely many forbidden minors,
also admits a characterization by finitely many forbidden topological minors.
\end{thm}

Furthermore, the list of forbidden topological minors may be derived from
the list of forbidden minors.

On the other hand, graph containment notions of {\em {\Smi}s} (defined in Section~\ref{sec:S-minors}),
{\em {\mami}s} (defined in \ref{sec:matching-minors}),
and {\em {\comi}s} (defined in \ref{sec:conformal-minors}) ---
listed here in increasing order of strictness --- play key roles in the theory of {\mcg}s,
and especially in the study of {\em {\PO}s} (defined in \ref{sec:applications}). The
latter two coincide whenever the ``minor'' is subcubic. That is,
a subcubic \mcg~$H$ is a \mami\ of a \mcg~$G$ if and only if $H$
is a \comi\ of~$G$. However, even in the restricted case when the ``minor'' is cubic,
both of these notions are stricter than that of {\Smi}s.

Interestingly,
in most known theorems characterizing various classes of {\mcg}s,
the ``forbidden minors'' happen to be cubic.
Our Main Theorem is a result similar to Theorem~\ref{thm:minors-vs-topological-minors}
for this restricted case.

\begin{thm}
\label{thm:conformal-minors-vs-S-minors-cubic}
{\sc [Main Theorem]}
Any class of {\mcg}s, say~$\mathcal{G}$, that is characterized by finitely many forbidden {\Smi}s that are cubic,
also admits a characterization by finitely many forbidden {\comi}s that are cubic as well.
\end{thm}

Our proof is constructive. In particular, one may derive the list
of forbidden {\em {\comi}s} (defined in Section~\ref{sec:conformal-minors})
from the list of forbidden {\em {\Smi}s}
(defined in Section~\ref{sec:S-minors}).

In the rest of this section, we provide the necessary background, and state
other results that lead to our proof of the Main Theorem (\ref{thm:conformal-minors-vs-S-minors-cubic}).

\subsection{Ear decompositions and {\comi}s}
\label{sec:conformal-minors}

A path is {\em odd} if it has an odd number of edges; otherwise, it is {\em even}.
As in the case of \mbox{$2$-connected} graphs, there is an \ed\ theory underlying {\mcg}s.

By a {\em removable single ear} of a \mcg~$G$, we mean an odd path, say~$P$, each of whose internal vertices is
of degree two (in~$G$) such that the graph $G-P$ is \mc. (Note that, unless $G$ is a cycle, the ends of $P$ have degree
three or more in~$G$.) It is well-known, and not too difficult to prove, that every bipartite \mcg,
except $K_2$, has a removable single ear. However, the same does not hold for nonbipartite {\mcg}s. For instance,
neither $K_4$ nor $\overline{C_6}$, has a removable single ear. (In order to deal with nonbipartite graphs,
we need the flexibility to simultaneously remove  ``two single ears''.)

By a {\em removable double ear} of a \mcg~$G$, we mean a pair of vertex-disjoint odd paths, say $(P,Q)$,
each of whose internal vertices is of degree two (in~$G$) such that: (i) neither $G-P$ nor $G-Q$ is \mc,
and (ii) $G-P-Q$ is \mc. For the sake of brevity, by a {\em removable ear},
we mean either a {\em removable single ear} or a {\em removable double ear}.
A fundamental result in matching theory (see \cite{lopl86}) states that every \mcg, except $K_2$,
has a removable ear.

To put it differently, given any \mcg~$G$, there exists a sequence $(G_1, G_2, \dots, G_r)$
of {\mc} subgraphs of~$G$ such that: (i) $G_1:=K_2$, (ii) $G_r := G$, and (iii) for each $2 \leq i \leq r$,
the graph $G_{i-1}$ is obtained from $G_i$ by deleting a removable ear.
We refer to such a sequence as an {\em \ed\ of the graph~$G$}.
We now proceed to state a further strengthening of this result.

A graph is {\em matchable} if it has a \pema,
and a subgraph~$H$ of a graph~$G$ is {\em conformal} if the graph $G-V(H)$ is matchable.
It is easily observed that every subgraph, in an \ed\ of a \mcg~$G$, is a conformal
subgraph of~$G$. The stronger version (mentioned above) says that, for any conformal \mc\ subgraph~$H$ of a
\mcg~$G$, there is an \ed, say $(G_1, G_2, \dots, G_r)$, of~$G$ such that $H$ appears
in this sequence. We now proceed to state a different strengthening applicable to nonbipartite
graphs that will immediately lead us to the notion of a {\em conformal minor}.

To {\em bi-subdivide} an edge means to subdivide it by
inserting an even number of subdivision vertices.
A graph~$H$ is a {\it \bisub} of a graph~$J$
if the former may be obtained from the latter by means of bi-subdividing
each edge in some subset of~$E(J)$. It is easily proved
that, for a \mcg~$J$ distinct from~$K_2$,
each \bisub\ of~$J$ is also \mc.
Lov{\'a}sz \cite{lova83} proved the following deep result.

\begin{thm}
\label{thm:nonbip-mcg-K4-prism-ear-dec}
{\sc [Lov{\'a}sz's Theorem]}
Every nonbipartite \mcg~$G$ admits an \ed, say~$(G_1, G_2, \dots, G_r)$,
such that either $G_3$ is a bi-subdivision of $K_4$, or otherwise $G_4$ is a bi-subdivision of~$\overline{C_6}$.
\end{thm}

A \mcg~$J$ is a {\it \comi} of a \mcg~$G$ if the
latter has a conformal subgraph~$H$ that is a \bisub\ of~$J$.
For the sake of convenience,
if a \mcg~$J$ is a \comi\ of a \mcg~$G$ then we say that $G$ is \JB;
otherwise, $G$ is \JF.
Lov{\'a}sz's Theorem (\ref{thm:nonbip-mcg-K4-prism-ear-dec}) may thus be restated as follows.

\begin{thm}
{\sc [Lov{\'a}sz's Theorem]}
\label{thm:nonbip-mcg-K4-prism-comi}
Every nonbipartite \mcg\ is either \KFB\ or \PB\ (or both).
\end{thm}

In Section~\ref{sec:theta-K4}, we will prove the following result that is reminescent of Lov{\'a}sz's Theorem, but is
easier to prove and applies to all {\mcg}s. (In fact, we will prove something much stronger
that will help us in proving our Main Theorem.)

\begin{thm}
\label{thm:mcg-theta-K4-comi}
{\sc [$\Theta$ - $K_4$ Theorem]}
Every \mcg\ is either \TB\ or \KFB\ (or both).
\end{thm}

We now introduce a notation that we will find useful later.
Given a finite set of {\mcg}s, say~$\mathcal{J}$, we say that a \mcg~$G$ is $\mathcal{J}$-free
if $G$ is \JF\ for each graph $J \in \mathcal{J}$; otherwise, we say that $G$ is $\mathcal{J}$-based.
For instance, Lov{\'a}sz's Theorem (\ref{thm:nonbip-mcg-K4-prism-comi})
states that every nonbipartite \mcg\ is $\{K_4,\overline{C_6}\}$-based,
whereas Theorem~\ref{thm:mcg-theta-K4-comi} states
that there is no $\{\Theta,K_4\}$-free \mcg.

Next, we briefly discuss `{\mami}s', introduced by Norine and Thomas~\cite{noth07},
and compare them to {\comi}s.

\subsection{Bi-contractions and {\mami}s}
\label{sec:matching-minors}

Bi-contracting a vertex~$v$ of degree two, with two distinct neighbors, is the operation
of contracting the two edges incident with it;
see Figure~\ref{fig:mami-need-not-imply-comi}.
It is easily observed that if $G$ is \mc,
then the graph resulting from any bi-contraction, is also \mc.

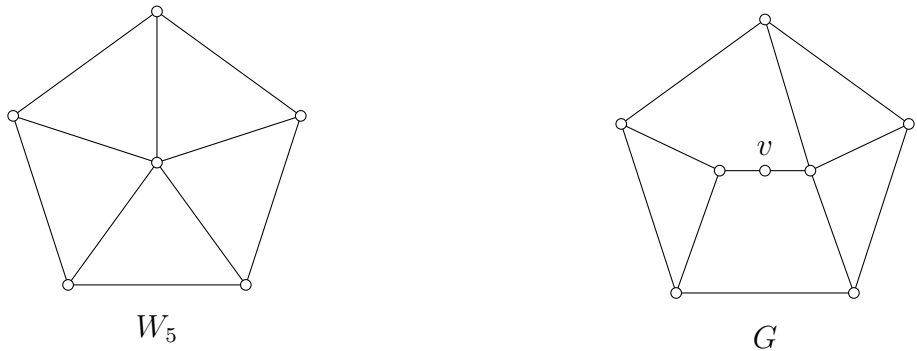
\begin{figure}[!htb]
\centering

\begin{tikzpicture}

\draw (90:2) -- (162:2) -- (234:2) -- (306:2) -- (18:2) -- (90:2);
\draw (0:0) -- (162:2);
\draw (0:0) -- (234:2);
\draw (0:0) -- (306:2);
\draw (0:0) -- (18:2);
\draw (0:0) -- (90:2);

\draw (90:2)node{};
\draw (162:2)node{};
\draw (234:2)node{};
\draw (306:2)node{};
\draw (18:2)node{};
\draw (0:0)node{};

\draw (270:2.2) node[nodelabel]{$W_5$};

\end{tikzpicture}
\hspace*{1.5in}
\begin{tikzpicture}

\draw (90:2) -- (162:2) -- (234:2) -- (306:2) -- (18:2) -- (90:2);
\draw (180:0.6) -- (0:0.6);
\draw (180:0.6) -- (162:2);
\draw (180:0.6) -- (234:2);
\draw (0:0.6) -- (306:2);
\draw (0:0.6) -- (18:2);
\draw (0:0.6) -- (90:2);

\draw (90:2)node{};
\draw (162:2)node{};
\draw (234:2)node{};
\draw (306:2)node{};
\draw (18:2)node{};
\draw (180:0.6)node{};
\draw (0:0.6)node{};
\draw (0:0)node{};

\draw (90:0.3)node[nodelabel]{$v$};

\draw (270:2.2) node[nodelabel]{$G$};

\end{tikzpicture}
\caption{The odd wheel $W_5$ is a \mami\ of~$G$ but $W_5$ is not a \comi\ of~$G$}
\label{fig:mami-need-not-imply-comi}
\end{figure}

A \mcg~$H$ is a {\em \mami} of a \mcg~$G$ if the former may be obtained from the latter
by any sequence of two types of operations: (i) deletion of a removable ear, and (ii) bi-contraction.
We now state an equivalent definition of {\comi}s.

Let $v$ be a vertex of degree two in a graph~$G$, with two distinct neighbors $u$~and~$w$,
one of which, say $u$, has degree two as well. Observe that if~$H$ is obtained from~$G$ by
\mbox{bi-contracting}~$v$ then $G$ is a bi-subdivision of~$H$; we refer to such a bi-contraction
as a {\em restricted bi-contraction}. (It is simply the reverse operation of bi-subdividing an edge
by inserting two subdivision vertices.)
This leads us the following observation.

\begin{prp}
A \mcg~$J$ is a \comi\ of a \mcg~$G$ if and only if $J$ may be obtained from~$G$
by a sequence of two types of operations: (i) deletion of a removable ear, and (ii) restricted bi-contraction. \qed
\end{prp}

Consequently, as mentioned earlier, {\mami}s is indeed a weaker notion of graph containment than
that of {\comi}s.

\begin{cor}
\label{cor:comi-implies-mami}
If a \mcg~$J$ is a \comi\ of a \mcg~$G$ then $J$ is a \mami\ of~$G$. \qed
\end{cor}

In general, the converse does not hold; see Figure~\ref{fig:mami-need-not-imply-comi}.
That being said, it is easily verified that the two notions coincide when the ``minor'' is subcubic.

\begin{prp}
A subcubic \mcg~$J$ is a \mami\ of a \mcg~$G$ if and only if~$J$ is a \comi\ of~$G$. \qed
\end{prp}

Lov{\'a}sz's Theorem (\ref{thm:nonbip-mcg-K4-prism-comi}) immediately leads to two natural problems:
characterize \KFF,
and likewise \PF, {\mcg}s. The planar case for each of these problems was solved
by Kothari and Murty~\cite{komu16}. In order to do so, they first ``reduced''
each of these problems (irrespective of planarity) to the case of `bricks' (which are
special types of nonbipartite {\mcg}s); we explain
their ``reduction'' in the next section since it is relevant to our approach for proving the
Main Theorem (\ref{thm:conformal-minors-vs-S-minors-cubic}).

\subsection{Tight cuts, bricks and braces}

For a graph $G$ and a subset $X \subseteq V(G)$, we use $\partial_G(X)$,
or simply $\partial(X)$ when there is no ambiguity,
to denote the corresponding cut
(that is, the set of edges with one end in $X$ and the other end in $\overline{X}:=V(G)-X$);
we refer to $X$ and $\overline{X}$ as the {\em shores} of the cut.
A cut is {\em trivial} if either of its shores is a singleton; otherwise, it is {\em nontrivial}.
For a trivial cut (i.e., edges incident at a vertex $v$),
we simplify notation as follows: $\partial(v) := \partial(\{v\})$.
By a {\em $k$-cut}, we mean a cut that comprises precisely $k$ edges.

For a cut $C:=\partial(X)$, we use $G/X \rightarrow x$, or simply $G/X$,
to denote the graph obtained from $G$ by shrinking the shore $X$ to a single vertex $x$; all edges
with both ends in $X$ are deleted. The graph $G/\overline{X}$ is defined analogously.
We refer to $G/X$ and $G/\overline{X}$ as the $C$-contractions of $G$.

A cut $C$ of a \mcg~$G$ is a {\em tight cut} if it meets each \pema~$M$ exactly once (i.e.,
$|M \cap C| = 1$). It may be easily verified that if $C$ is a tight cut of a \mcg~$G$
then each $C$-contraction of~$G$ is \mc\ as well.
A \mcg\ devoid of nontrivial tight cuts is called a {\em brick} if it is nonbipartite, and a {\em brace} otherwise.
For instance: $\Theta, C_4$ and $K_{3,3}$ are braces whereas $K_4$ and $\overline{C_6}$ are bricks.

Using tight cuts, one may now define the well-known {\em tight cut decomposition procedure} --- for any
matching covered graph $G$ --- as follows. If $G$ has a nontrivial tight cut, say~$C$, then we obtain
two smaller {\mcg}s by considering the $C$-contractions of~$G$;
if either of them has a nontrivial tight cut then we
repeat the same process to obtain even smaller {\mcg}s. We do this recursively until we obtain a list of
bricks and braces. Note that, since one may choose any nontrivial tight cut at each step, a \mcg\
may admit many different applications of the tight cut decomposition procedure.
However, Lov{\'a}sz~\cite{lova87} proved the following remarkable result.

\begin{thm}
{\sc [Unique Tight Cut Decomposition Theorem]}
Any two applications of the tight cut decomposition procedure on a matching covered graph~$G$
yield the same list of bricks and braces (up to multiplicities of edges).
\end{thm}

In light of the above theorem, for any \mcg~$G$, we may refer to ---
the underlying simple graphs produced by any application of the tight cut decomposition procedure on~$G$ ---
as the {\em bricks and braces of $G$}.
It is worth noting that nontrivial tight cuts may be discovered in polynomial-time; this follows, for instance,
from the work of Edmonds, Lov{\'a}sz and Pulleyblank \cite{elp82}.
Consequently, the bricks and braces of a \mcg\
may be computed in polynomial-time.

A graph $G$ is {\em bicritical} if the graph~$G-u-v$ is matchable for each pair $u,v$ of distinct vertices.
They \cite{elp82} also proved the following characterization of bricks that we will find useful later.

\begin{thm}
\label{thm:elp-brick-characterization}
A graph~$G$, of order four or more, is a brick if and only if $G$ is \mbox{$3$-connected} and bicritical.
\end{thm}

The significance of the tight cut decomposition procedure arises from the following phenomenon:
for many properties of interest, whether or not a \mcg~$G$ has the desired property
depends entirely on whether its bricks and braces have the desired property.
This was exploited by Kothari and Murty~\cite{komu16} in the following manner.

\begin{thm}
\label{thm:cubic-brick-tight-cut-reduction}
Let $J$ denote any cubic brick. For any tight cut~$C$ of a \mcg~$G$,
the graph~$G$ is \JF\ if and only if both of its $C$-contractions are \JF.
\end{thm}

It follows from Theorem~\ref{thm:cubic-brick-tight-cut-reduction}
that a \mcg~$G$ is \KFF\ if and only if each of its bricks is \KFF; likewise, for {\PF}ness.
Figure~\ref{fig:cubic-brick-tight-cut-reduction} shows two examples of {\mcg}s --- each of which
has a unique nontrivial tight cut indicated by the bold line;
their {\PF}ness and {\KFF}ness, respectively, may be explained using Theorem~\ref{thm:cubic-brick-tight-cut-reduction}.

\begin{figure}[!htb]
\centering
\begin{tikzpicture}

\draw[ultra thick] (-0.5,-0.75) -- (3.5,-0.75);
\draw (-0.75,-0.75)node[nodelabel]{$C$};

\draw (0.75,1) -- (0,0) -- (2.25,1) -- (1.5,0) -- (0.75,1) -- (3,0) -- (2.25,1);
\draw (1.5,-1.5) -- (0.75,-2.5) -- (2.25,-2.5) -- (1.5,-1.5);
\draw (0.75,-2.5) -- (0,0);
\draw (1.5,-1.5) -- (1.5,0);
\draw (2.25,-2.5) -- (3,0);

\draw (0,0)node{};
\draw (1.5,0)node{};
\draw (3,0)node{};

\draw (0.75,1)node{};
\draw (2.25,1)node{};

\draw (1.5,-1.5)node{};
\draw (0.75,-2.5)node{};
\draw (2.25,-2.5)node{};

\draw (1.5,-3)node[nodelabel]{(a)};

\end{tikzpicture}
\hspace*{1in}
\begin{tikzpicture}

\draw[ultra thick] (-0.5,-0.75) -- (3.5,-0.75);
\draw (3.75,-0.75)node[nodelabel]{$C$};

\draw (0.75,1) -- (0,0) -- (2.25,1) -- (1.5,0) -- (0.75,1) -- (3,0) -- (2.25,1);
\draw (1.5,-1.5) -- (0.5,-2.5);
\draw (1.5,-1.5) -- (2.5,-2.5);
\draw (2.5,-2.5) -- (0.5,-2.5);
\draw (1,-2) -- (2,-2);
\draw (0.5,-2.5) -- (0,0);
\draw (1.5,-1.5) -- (1.5,0);
\draw (2.5,-2.5) -- (3,0);

\draw (0,0)node{};
\draw (1.5,0)node{};
\draw (3,0)node{};

\draw (0.75,1)node{};
\draw (2.25,1)node{};

\draw (0.5,-2.5)node{};
\draw (2.5,-2.5)node{};
\draw (1,-2)node{};
\draw (2,-2)node{};
\draw (1.5,-1.5)node{};

\draw (1.5,-3)node[nodelabel]{(b)};

\end{tikzpicture}
\caption{(a) A \PF\ \mcg; (b) A \KFF\ \mcg}
\label{fig:cubic-brick-tight-cut-reduction}
\end{figure}

It is worth noting that neither implication of Theorem~\ref{thm:cubic-brick-tight-cut-reduction}
holds if one replaces `cubic brick' by `$2$-connected cubic graph'. For instance,: for the forward implication,
it may be easily verified that the graph shown in Figure~\ref{fig:cubic-brick-tight-cut-reduction}a is \KTTF,
whereas one of its \mbox{$C$-contractions} is \KTTB;
for the reverse implication,
if $G$ is the graph shown in Figure~\ref{fig:cubic-brick-tight-cut-reduction}b, then clearly it is $G$-based,
whereas each of its \mbox{$C$-contractions} is $G$-free.

Having reduced both of the problems (i.e., of characterizing \KFF, and \PF, {\mcg}s) to the case of bricks
(via Theorem~\ref{thm:cubic-brick-tight-cut-reduction}),
Kothari and Murty~\cite{komu16} then characterized the \KFF\ planar bricks
as follows.

\begin{thm}
\label{thm:K4-free-planar-bricks}
A planar brick is \KFF\ if and only if its planar embedding has precisely two faces bounded
by cycles of odd length.
\end{thm}

Note that, by Theorem~\ref{thm:elp-brick-characterization}, bricks are $3$-connected; whence the planar ones
have a unique planar embedding. (We omit their characterization of \PF\ planar bricks.)

Next, we describe a generalization of tight cuts known as `separating cuts', and
revisit Theorem~\ref{thm:cubic-brick-tight-cut-reduction}.

\subsection{Separating cuts and {\Smi}s}

A cut~$C$ of a \mcg~$G$ is a {\em separating cut} if both $C$-contractions are also \mc.
It follows immediately that every tight cut is a separating cut; the converse is not true.
The following is easily proved; see \cite[Proposition 3]{kcll20}.

\begin{prp}
\label{prp:cubic-mcg-3cuts-are-separating-cuts}
In a $2$-connected cubic graph, each $3$-cut is a separating cut.
\end{prp}

\begin{figure}[!htb]
\centering
\begin{tikzpicture}

\draw[ultra thick] (1.5,-0.5) -- (1.5,2.5);
\draw (1.5,2.75)node[nodelabel]{$C$};

\draw (0,0) -- (1,1) -- (0,2) -- (0,0);
\draw (5,0) -- (4,1) -- (5,2) -- (5,0);
\draw (1,1) -- (4,1);
\draw (0,0) -- (5,0);
\draw (0,2) -- (5,2);
\draw (2,1) -- (2,2);
\draw (3,1) -- (3,2);

\draw (0,0)node{};
\draw (1,1)node{};
\draw (0,2)node{};
\draw (5,0)node{};
\draw (4,1)node{};
\draw (5,2)node{};
\draw (2,1)node{};
\draw (2,2)node{};
\draw (3,1)node{};
\draw (3,2)node{};

\draw (2.5,-1)node[nodelabel]{(a)};

\end{tikzpicture}
\hspace*{1in}
\begin{tikzpicture}

\draw[ultra thick] (2.5,-0.5) -- (2.5,2.5);
\draw (2.5,2.75)node[nodelabel]{$C$};

\draw (0,0) -- (1,1) -- (0,2) -- (0,0);
\draw (5,0) -- (4,1) -- (5,2) -- (5,0);
\draw (1,1) -- (4,1);
\draw (0,0) -- (5,0);
\draw (0,2) -- (5,2);
\draw (2,1) -- (2,2);
\draw (3,1) -- (3,0);

\draw (0,0)node{};
\draw (1,1)node{};
\draw (0,2)node{};
\draw (5,0)node{};
\draw (4,1)node{};
\draw (5,2)node{};
\draw (2,1)node{};
\draw (2,2)node{};
\draw (3,1)node{};
\draw (3,0)node{};

\draw (2.5,-1)node[nodelabel]{(b)};

\end{tikzpicture}
\caption{(a) A \KFF\ brick whose both $C$-contractions are \KFB; (b) A \KFB\ brick whose both
$C$-contractions are \KFF}
\label{fig:cubic-brick-separating-cut-reduction-fails}
\end{figure}
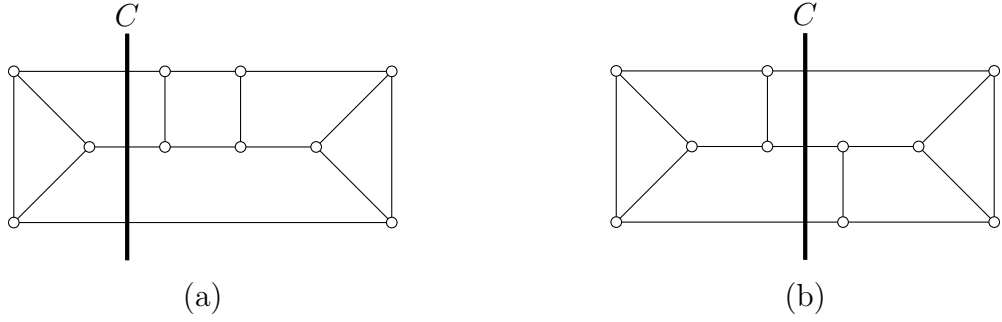

Once again, we note that neither implication of~Theorem~\ref{thm:cubic-brick-tight-cut-reduction}
holds if one replaces `tight cut' by `separating cut'.
Figure~\ref{fig:cubic-brick-separating-cut-reduction-fails}
shows an example (one for each implication); the bold line indicates a separating cut.
(The reader may find Theorem~\ref{thm:K4-free-planar-bricks} useful in order to verify the properties.)

Interestingly, however,
one can prove a similar result (only one implication) that applies to separating cuts, and to all $2$-connected cubic graphs.
We state this result (Theorem~\ref{thm:K4-decoration}) in Section~\ref{sec:K4-decorations} after
introducing the necessary terminology.

Before that, let us proceed to define our final (and most relaxed) graph containment notion.
A \mcg~$J$ is a {\em separation-deletion minor}, or simply an {\em $S$-minor},
of a \mcg~$G$ if the former may be obtained from the latter by a sequence of two types
of operations: (i) deletion of a removable ear, and (ii) contraction of a shore of a separating cut.

Observe that `bi-contracting a vertex of degree two (with two distinct neighbors)'
is a special case of `contraction of a shore of a separating cut'.
This implies the following.

\begin{prp}
\label{prp:mami-implies-Smi}
If a \mcg~$J$ is a \mami\ of a \mcg~$G$ then $J$ is an \Smi\ of~$G$. \qed
\end{prp}

It follows from Corollary~\ref{cor:comi-implies-mami} that if a \mcg~$J$ is a \comi\ of a \mcg~$G$
then $J$ is an \Smi\ of~$G$. As we have noted earlier, in case the ``minor'' is subcubic, the
notion of {\comi}s coincides with that of {\mami}s. However, this is not the case with {\Smi}s.

For instance: (i) it follows from Lov{\'a}sz's Theorem (\ref{thm:nonbip-mcg-K4-prism-comi})
that $K_4$ is an \Smi\ of every nonbipartite \mcg\ (since $K_4$ is an \Smi\ of~$\overline{C_6}$);
however, $K_4$ is definitely not a \comi\ of every nonbipartite \mcg;
(ii) as mentioned earlier, $K_{3,3}$ is not a \comi\ of 
the graph shown in Figure~\ref{fig:cubic-brick-tight-cut-reduction}a; however, $K_{3,3}$
is an \Smi. (Interestingly, $K_{3,3}$ is a \comi\ of the graph shown in
Figure~\ref{fig:cubic-brick-tight-cut-reduction}.)

The notion of {\Smi}s was introduced by Carvalho, Lucchesi and Murty~\cite{clm12}.
On the other hand, for `weak {\mami}s': the notion was introduced by
Fischer, Little and Rendl \cite{fili01,lrf02},
and the term by Norine and Thomas \cite{noth08};
however, they did not restrict themselves to {\mcg}s.
As mentioned in
\cite{clm12}\footnote{It seems that there is a typo in \cite{clm12}. They cite \cite{noth07} instead of \cite{noth08}.},
the two notions (`weak {\mami}s' and `{\Smi}s') coincide when restricted to {\mcg}s.

Next, we describe the operation of `splicing' which is the reverse of the operation:
`contraction of a shore of a separating cut'. We refer the reader to \cite{lckm18} for
further discussion on the topic.

\subsection{Splicing and {\KFdec}s}
\label{sec:K4-decorations}

For $i \in \{1,2\}$, let $G_i$ denote a matching covered graph with a specified vertex $v_i$ such
that the degree of $v_1$ (in $G_1$) is the same as the degree of $v_2$ (in $G_2$); furthermore,
let $\pi$ denote a bijection between $\partial_{G_1}(v_1)$ and $\partial_{G_2}(v_2)$.
We now define the {\em splicing of $G_1$ and $G_2$, at $v_1$ and $v_2$ respectively, as per the bijection $\pi$},
denoted by $(G_1 \odot G_2)_{v_1,v_2,\pi}$, as the following graph $G$. The graph $G$
is obtained from $(G_1 - v_1) \cup (G_2 - v_2)$ by joining, for each edge $e$ in $\partial_{G_1}(v_1)$,
the end of $e$ in $V(G_1)-v_1$ with the end of $\pi(e)$ in $V(G_2)-v_2$. It is easily observed
that the graph $G$ is also \mc.

Note that $V(G)=(V(G_1)-v_1) \cup (V(G_2)-v_2)$; we refer to $C:=\partial_G(V(G_1)-v_1)$
as the {\em corresponding splicing cut}.
Observe that the $C$-contractions of $G$ are precisely $G_1$~and~$G_2$.
This proves the first part of the following statement; the second part holds by definition of {\Smi}s.

\begin{prp}
\label{prp:splicing-and-S-minors}
Let $G$ denote a \mcg\ obtained by splicing two {\mcg}s, say $G_1$~and~$G_2$,
and let $C$ denote the corresponding splicing cut.
Then $C$ is a separating cut of~$G$, and each of $G_1$ and $G_2$ is an \Smi\ of~$G$. \qed
\end{prp}

It is evident that the splicing of two graphs, at specified vertices, depends on the chosen bijection.
For instance, the pentagonal prism as well as the Petersen graph may both be obtained by splicing two
copies of the odd wheel $W_5$ (shown in Figure~\ref{fig:mami-need-not-imply-comi}) at their hubs.

We will be particularly interested in splicing cubic graphs where one of the two graphs being spliced is $K_4$. Observe that,
in this special case, the chosen bijection does not matter.
In fact, splicing any cubic graph $G_1$ and $G_2:=K_4$, at vertices $v_1$ and $v_2$,
is the same as replacing $v_1$ in $G_1$
by a triangle (i.e., $3$-cycle) in order to obtain another cubic graph $G$
(which is uniquely determined by $G_1$ and $v_1$).

For instance, the graph shown in Figure~\ref{fig:cubic-brick-tight-cut-reduction}a
is the (unique) result of splicing $K_{3,3}$ and $K_4$,
and we denote it simply by $K_{3,3} \odot K_4$.
That being said,
unless the graph~$G_1$ is vertex-transitive, the choice of the vertex~$v_1$ matters.
For instance, splicing $K_{3,3} \odot K_4$ and $K_4$ may lead to two
nonisomorphic graphs (depending on the choice of~$v_1$):
either the graph shown in Figure~\ref{fig:cubic-brick-tight-cut-reduction}b,
or the graph (of order~$10$) shown in Figure~\ref{fig:solid-K4-decorations-of-K33}.

Now, let $J$ be a $2$-connected cubic graph.
For each (not necessarily proper) subset~$T$ of~$V(J)$, we may obtain a (unique) graph,
by replacing each vertex in~$T$ by a triangle.
We denote this graph by $J^T \odot K_4$, and call it a {\em \KFdec\ of~$J$} --- since it
may instead be obtained from $J$ by splicing with copies of $K_4$ repeatedly (at each of the vertices in~$T$).
The latter viewpoint immediately leads us to the following consequence of
Proposition~\ref{prp:splicing-and-S-minors}.

\begin{cor}
  \label{cor:K4-decorations-and-S-minors}
\breaksta
  A $2$-connected cubic graph is an \Smi\ of each of its {\KFdec}s.~\qed
\end{cor}

Given a $2$-connected cubic graph~$J$, a fixed set $T \subseteq V(J)$,
and the corresponding \KFdec\ $L:=J^T \odot K_4$,
we may regard the set $\overline{T}:=V(J)-T$ to be a subset of~$V(L)$; we refer to these as the
{\em original vertices of~$L$}, and we refer to the remaining $3|T|$ vertices as {\em decoration vertices of~$L$}.
Furthermore, we may fix a three-to-one correspondence between the decoration vertices of~$L$ and the set $T$,
say $\tau:V(L)-\overline{T} \rightarrow T$.
For instance, if $J:=K_{3,3}$ and $T:=\{a_1,a_3\}$, as per the labeling in Figure~\ref{fig:solid-K4-decorations-of-K33},
then $L$ has four original vertices (namely, $b_1,b_2,b_3$ and $a_2$)
and six decoration vertices; the six decoration vertices induce two copies of~$C_3$ and we may fix a correspondence~$\tau$
so that the vertices of one copy of $C_3$ correspond to $a_1$, and the vertices of the other copy of~$C_3$
correspond to $a_3$.

By the {\em list of {\KFdec}s} of a $2$-connected cubic graph~$J$, denoted by $K_4(J)$,
we mean the set of all non-isomorphic $K_4$-decorations of~$J$.
In particular, $K_4(J) := \{ J^T \odot K_4 : T \subseteq V(J)\}$
is a finite set of $2$-connected cubic graphs.
Observe that, in theory, the set $K_4(J)$ may contain $2^{|V(J)|}$ graphs;
however, in practice, the graphs (i.e., $J$) typically encountered have many symmetries,
and the set $K_4(J)$ turns out to be much smaller.

For instance, using the symmetries of $K_{3,3}$, the reader may easily verify
that the set $K_4(K_{3,3})$ comprises only ten graphs;
four of these are shown in Figure~\ref{fig:solid-K4-decorations-of-K33}. (For a bipartite graph~$G$,
we use the notation $G[A,B]$ to denote its color classes $A$~and~$B$.)
Henceforth, for $K_{3,3}[A,B]$, we shall
use the vertex-labeling depicted in Figure~\ref{fig:solid-K4-decorations-of-K33}.

\begin{figure}[!htb]
\centering
\begin{tikzpicture}

\draw (0,0) -- (0,2);
\draw (0,0) -- (2,2);
\draw (0,0) -- (4,2);

\draw (2,0) -- (0,2);
\draw (2,0) -- (2,2);
\draw (2,0) -- (4,2);

\draw (4,0) -- (0,2);
\draw (4,0) -- (2,2);
\draw (4,0) -- (4,2);

\draw (0,0)node{}node[nodelabel,below]{$a_1$};
\draw (2,0)node{}node[nodelabel,below]{$a_2$};
\draw (4,0)node{}node[nodelabel,below]{$a_3$};

\draw (0,2)node{}node[nodelabel,above]{$b_1$};
\draw (2,2)node{}node[nodelabel,above]{$b_2$};
\draw (4,2)node{}node[nodelabel,above]{$b_3$};

\draw (2,-1.5)node[nodelabel]{$K_{3,3}$};

\end{tikzpicture}
\hspace*{1in}
\begin{tikzpicture}

\draw (0,0) -- (0,2);
\draw (0,0) -- (2,2);
\draw (0,0) -- (4,2);

\draw (2,0.3) -- (1.6,-0.3) -- (2.4,-0.3) -- (2,0.3);
\draw (1.6,-0.3) -- (0,2);
\draw (2,0.3) -- (2,2);
\draw (2.4,-0.3) -- (4,2);

\draw (4,0) -- (0,2);
\draw (4,0) -- (2,2);
\draw (4,0) -- (4,2);

\draw (0,0)node{};
\draw (4,0)node{};

\draw (2,0.3)node{};
\draw (1.6,-0.3)node{};
\draw (2.4,-0.3)node{};

\draw (0,2)node{};
\draw (2,2)node{};
\draw (4,2)node{};

\draw (2,-1.5)node[nodelabel]{$K_{3,3} \odot K_4$};

\end{tikzpicture}

\vspace*{0.2in}

\begin{tikzpicture}

\draw (0,0.3) -- (-0.4,-0.3) -- (0.4,-0.3) -- (0,0.3);
\draw (-0.4,-0.3) -- (0,2);
\draw (0,0.3) -- (2,2);
\draw (0.4,-0.3) -- (4,2);

\draw (2,0) -- (0,2);
\draw (2,0) -- (2,2);
\draw (2,0) -- (4,2);

\draw (4,0.3) -- (3.6,-0.3) -- (4.4,-0.3) -- (4,0.3);
\draw (3.6,-0.3) -- (0,2);
\draw (4,0.3) -- (2,2);
\draw (4.4,-0.3) -- (4,2);

\draw (0,0.3)node{};
\draw (-0.4,-0.3)node{};
\draw (0.4,-0.3)node{};

\draw (2,0)node{};

\draw (4,0.3)node{};
\draw (3.6,-0.3)node{};
\draw (4.4,-0.3)node{};

\draw (0,2)node{};
\draw (2,2)node{};
\draw (4,2)node{};

\draw (2,-1.5)node[nodelabel]{$K_{3,3}^{\{a_1,a_3\}} \odot K_4$};
\end{tikzpicture}
\hspace*{1in}
\begin{tikzpicture}

\draw (0,0.3) -- (-0.4,-0.3) -- (0.4,-0.3) -- (0,0.3);
\draw (-0.4,-0.3) -- (0,2);
\draw (0,0.3) -- (2,2);
\draw (0.4,-0.3) -- (4,2);

\draw (2,0.3) -- (1.6,-0.3) -- (2.4,-0.3) -- (2,0.3);
\draw (1.6,-0.3) -- (0,2);
\draw (2,0.3) -- (2,2);
\draw (2.4,-0.3) -- (4,2);

\draw (4,0.3) -- (3.6,-0.3) -- (4.4,-0.3) -- (4,0.3);
\draw (3.6,-0.3) -- (0,2);
\draw (4,0.3) -- (2,2);
\draw (4.4,-0.3) -- (4,2);

\draw (0,0.3)node{};
\draw (-0.4,-0.3)node{};
\draw (0.4,-0.3)node{};

\draw (2,0.3)node{};
\draw (1.6,-0.3)node{};
\draw (2.4,-0.3)node{};

\draw (4,0.3)node{};
\draw (3.6,-0.3)node{};
\draw (4.4,-0.3)node{};

\draw (0,2)node{};
\draw (2,2)node{};
\draw (4,2)node{};

\draw (2,-1.5)node[nodelabel]{$K_{3,3}^{A} \odot K_4$};
\end{tikzpicture}
\vspace*{-0.4in}
\caption{Four {\KFdec}s of $K_{3,3}[A,B]$ where $A:=\{a_1,a_2,a_3\}$ and $B:=\{b_1,b_2,b_3\}$}
\label{fig:solid-K4-decorations-of-K33}
\end{figure}
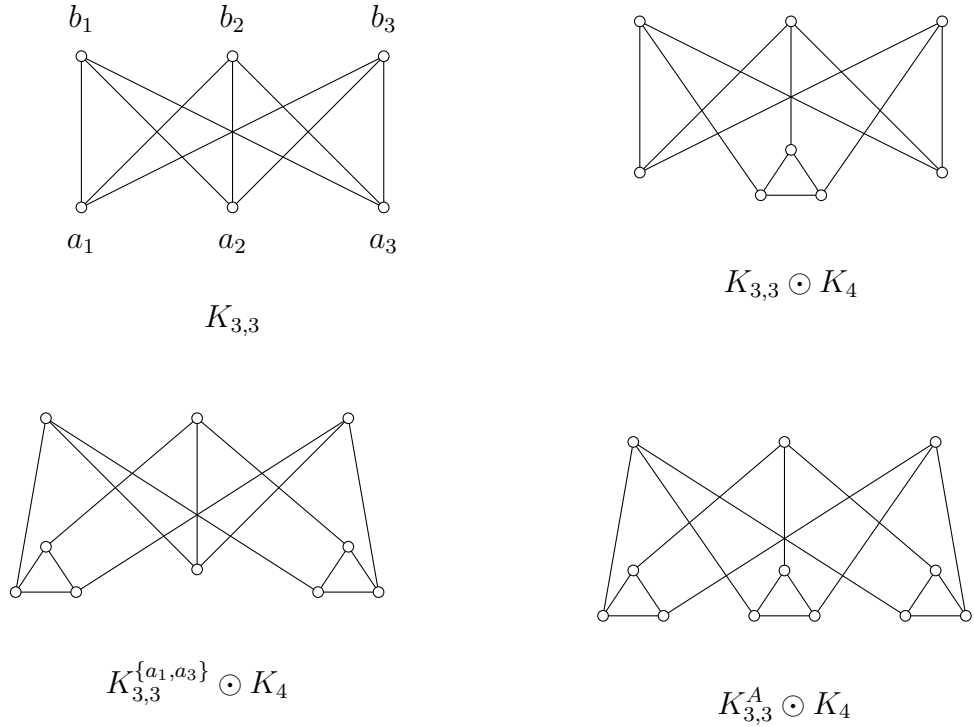

Recall that, for a finite set of {\mcg}s, say~$\mathcal{J}$, we say that a \mcg~$G$ is $\mathcal{J}$-free
if $G$ is \JF\ for each graph $J \in \mathcal{J}$; otherwise, $G$ is $\mathcal{J}$-based.
We are now ready to state the promised result that is similar to (one direction of)
Theorem~\ref{thm:cubic-brick-tight-cut-reduction}, but applies to all $2$-connected cubic graphs (instead of just cubic bricks)
and to separating cuts (instead of just tight cuts).

\begin{thm}
\label{thm:K4-decoration}
{\sc [$K_4$-decoration Theorem]}
Let $J$ denote any $2$-connected cubic graph.
For any separating cut $C$ of a \mcg~$G$,
if the graph $G$ is $K_4(J)$-free
then both of its $C$-contractions are $K_4(J)$-free.
\end{thm}

The reader may find it instructive to verify the above statement for some of the examples
shown in Figures~\ref{fig:cubic-brick-tight-cut-reduction}~and~\ref{fig:solid-K4-decorations-of-K33}.
For instance, if $G$ is the graph shown in Figure~\ref{fig:cubic-brick-tight-cut-reduction}b,
and if $X$ is the (unique) set of vertices such that the induced subgraph $G[X]$ is isomorphic to $C_3$,
then $G/X$ is $K_4(K_{3,3})$-based; consequently, $G$ is also $K_4(K_{3,3})$-based.
(Note that $\partial(X)$ is a separating cut in~$G$ that is not tight.)

A proof of the \KFdec\ Theorem, using double induction,
was found by one of the authors, Cl{\'a}udio L. Lucchesi, in 2012; this has not been published anywhere.
In Section~\ref{sec:K4-decoration-thm-proof},
we present a different and direct proof of the \KFdec\ Theorem
that relies on the \mbox{$\Theta$ - $K_4$} Theorem (\ref{thm:mcg-theta-K4-comi}).

Our Main Theorem (\ref{thm:conformal-minors-vs-S-minors-cubic}) relies heavily on the \KFdec\ Theorem.
In the next section, we prove it (assuming the \KFdec\ Theorem).

\subsection{Proof of the Main Theorem}
\label{sec:S-minors}

Before proving the Main Theorem (which applies to a finite set of $2$-connected cubic graphs), we
prove the key lemma that applies to a fixed $2$-connected cubic graph.

\begin{lem}
\label{lem:S-minor-implies-decorated-conformal-minor}
Let $H$ denote any $2$-connected cubic graph, and let $G$ be a \mcg. If $H$ is an \Smi\ of~$G$
then some \KFdec\ of~$H$ is a \comi\ of~$G$.
\end{lem}
\begin{proof}
Assume that $H$ is an \Smi\ of~$G$. Thus there exists a sequence $(G_1, G_2, \dots, G_r)$,
where $G_1:=H$ and $G_r:=G$, so that, for each $2 \leq i \leq r$,
the graph~$G_{i-1}$ is obtained from $G_i$ either by deleting a removable ear,
or by contracting the shore of a separating cut.
Note that $H$ is an \Smi\ of each graph in the sequence.
If $r=1$ then $H$ is isomorphic to~$G$; ergo a \comi\ of~$G$.

Now suppose that $r \geq 2$, and
assume inductively that the desired conclusion holds for $G_{r-1}$. Thus,
some \KFdec\ of~$H$, say~$J$, is a \comi\ of~$G_{r-1}$. If $G_{r-1}$ is obtained from $G_r$
by deleting a removable ear then clearly $J$ is a \comi\ of~$G_r$.
Otherwise, $G_{r-1}$ is obtained from~$G_r$ by contracting the shore of a separating cut, say~$C$;
in particular, $G_{r-1}$ is a $C$-contraction of~$G_r$.
By the \KFdec\ Theorem (\ref{thm:K4-decoration}),
some \KFdec\ of~$H$, say $L$, is a \comi\ of~$G_r$.
In either case, we have shown that some \KFdec\ of~$H$ is a \comi\ of~$G$.
\end{proof}

We are now ready to prove the Main Theorem (\ref{thm:conformal-minors-vs-S-minors-cubic}).
In fact, we prove a refined version (\ref{thm:conformal-minors-vs-S-minors-cubic-long-version})
that also explains how the finite list of forbidden {\comi}s
is derived from the finite list of forbidden {\Smi}s. In order to do so, we generalize the notation
$K_4(\cdot)$ from a single graph to a finite set of graphs. For a finite set of $2$-connected cubic graphs,
say $\mathcal{H}$, we define $K_4(\mathcal{H})$ as follows.

\[
K_4(\mathcal{H}) := \displaystyle\bigcup_{H \in \mathcal{H}} K_4(H)
\]

\begin{thm}
\label{thm:conformal-minors-vs-S-minors-cubic-long-version}
{\sc [Main Theorem]}
Let $\mathcal{H}$ denote a finite set of $2$-connected cubic graphs, and
let $\mathcal{G}$ denote a class of {\mcg}s such that $G \in \mathcal{G}$
if and only if, for each $H \in \mathcal{H}$, the graph $H$ is not an \Smi\ of~$G$.
Then $G \in \mathcal{G}$ if and only if $G$ is $K_4(\mathcal{H})$-free.
\end{thm}
\begin{proof}
We let $\mathcal{H}$ and $\mathcal{G}$ be as described in the theorem statement.

First suppose that $G$ is $K_4(\mathcal{H})$-based; i.e.,
there exists a graph $J \in K_4(\mathcal{H})$ that is a \comi\ of~$G$.
By Proposition~\ref{prp:mami-implies-Smi}, $J$ is an \Smi\ of~$G$.
By definition, $J$ is a \KFdec\ of some graph $H \in \mathcal{H}$.
By Corollary~\ref{cor:K4-decorations-and-S-minors}, $H$ is an \Smi\ of~$J$.
It follows from transitivity that $H$ is an \Smi\ of~$G$. Thus $G \notin \mathcal{G}$.

Conversely, suppose that $G \notin \mathcal{G}$. By hypothesis, there exists a graph $H \in \mathcal{H}$
that is an \Smi\ of~$G$. By Lemma~\ref{lem:S-minor-implies-decorated-conformal-minor},
some \KFdec\ of~$H$, say~$J$, is a \comi\ of~$G$. By definition, $J \in K_4(\mathcal{H})$;
whence $G$ is $K_4(\mathcal{H})$-based.

This completes the proof of Theorem~\ref{thm:conformal-minors-vs-S-minors-cubic-long-version}.
\end{proof}

Observe that, for a finite set of $2$-connected cubic graphs, say~$\mathcal{H}$,
the set $K_4(\mathcal{H})$ has at most
$\sum_{H \in \mathcal{H}} 2^{|V(H)|}$ graphs.
Consequently, in theory, the number of forbidden {\comi}s may be exponential
in terms of the number of forbidden {\Smi}s. However, as we will see in Section~\ref{sec:applications},
in practice, the numbers are somewhat comparable due to two reasons (as follows). Firstly, most
of the forbidden minors encountered have various symmetries;
consequently, many of their {\KFdec}s turn out to be isomorphic.
Secondly, the particular class of graphs (one seeks to characterize) may impose additional constraints, and
often allows one to omit many of the {\KFdec}s;
this point will become clearer in Section~\ref{sec:applications} via concrete examples.

\subsection{Organization of this paper}

In Section~\ref{sec:proofs}, we prove the \KFdec\ Theorem (\ref{thm:K4-decoration}).
In order to prove this, we will first characterize {\mcg}s that have precisely three {\pema}s,
and use it to deduce the $\Theta$ - $K_4$ Theorem (\ref{thm:mcg-theta-K4-comi}).

In Section~\ref{sec:applications}, we will provide applications of our Main Theorem
to the theory of \pfaff\ orientations. In particular, we deduce known forbidden \comi\
characterizations of \pfaff\ `near-bipartite' graphs, and of \pfaff\ `solid' graphs, using their
respective known forbidden \Smi\ characterizations.

\section{Proofs}
\label{sec:proofs}

\subsection{Graphs with three perfect matchings}

Observe that each of $\Theta$ and $K_4$ has precisely three {\pema}s; the same holds for their {\bisub}s.
In this section, we first present a proof of the converse --- these are the only {\mcg}s that have precisely three {\pema}s.
To this end, the following observation will come in handy.

\begin{lem}
\label{lem:even-conformal-bicycle}
If a graph~$G$ has two vertex-disjoint even cycles $C_1$~and~$C_2$ such that $C_1 \cup C_2$ is a conformal subgraph,
then $G$ has at least four {\pema}s. \qed
\end{lem}

\smallskip
Let $G$ be a \mcg, and let $v_0$ denote a vertex of degree two that has two distinct neighbors, say~$v_1$~and~$v_2$.
Let $H$ denote the graph obtained from $G$ by contracting the two edges incident with $v_0$; we say that $H$ is obtained form $G$ by {\em bi-contracting} $v_0$. Note that $\partial(X)$, where $X:=\{v_0,v_1,v_2\}$, is a tight cut of $G$; consequently, $H$ is \mc. It is easily verified that the {\pema}s of $G$, and those of~$H$, are in bijective correspondence.

\smallskip
Given a \mcg~$G$, one may repeatedly apply the bi-contraction operation in order to obtain a \mcg\ that
is devoid of degree two vertices with two neighbors; this graph, denoted $\widehat{G}$, is uniquely determined up to isomorphism (see \cite[Proposition 3.11]{clm05}) and is called the {\em retract} of~$G$.
This, coupled with the observation from the preceding paragraph, yields the following.

\begin{prp}
\label{prp:pema-correspondence-retract}
For every \mcg~$G$, there is a bijective correspondence between the sets of {\pema}s of~$G$ and those of $\widehat{G}$.
Consequently, $G$ and $\widehat{G}$ have an equal number of {\pema}s. \qed
\end{prp}

Observe that, for each \mcg~$G$ distinct from $K_2$ and cycles, its retract~$\widehat{G}$ has minimum degree three or more. Furthermore, one may easily verify that $\widehat{G}$ is cubic if and only if $G$ is a \bisub\ of~$\widehat{G}$.

\smallskip
Now, let $G$ be a \mcg\ with precisely three {\pema}s. By Proposition~\ref{prp:pema-correspondence-retract},
$\widehat{G}$ has precisely three {\pema}s. It follows from the preceding paragraph that $\widehat{G}$ is cubic;
thus $G$ is a \bisub\ of~$\widehat{G}$. This proves the following.

\begin{cor}
\label{cor:mcg-three-pemas-bisub-cubic}
Every \mcg\ that has precisely three {\pema}s is a \bisub\ of a $2$-connected cubic graph. \qed
\end{cor}

Now, we may proceed to prove the main result of this section.

\begin{thm}
\label{thm:mcg-3-pema}
A {\mcg} $G$ has precisely three {\pema}s if and only if $G$ is a \bisub\ of one of $\Theta$ and $K_4$.
\end{thm}
\begin{proof}
Let $J$ denote a $2$-connected cubic graph that has precisely three {\pema}s, say $M_1$, $M_2$ and $M_3$.
By Corollary~\ref{cor:mcg-three-pemas-bisub-cubic}, it suffices to show that $J$ is one of $\Theta$~and~$K_4$.

\smallskip
Since $J$ is cubic,
$M_1, M_2$ and $M_3$ are pairwise disjoint. Furthermore, by Lemma~\ref{lem:even-conformal-bicycle}, we infer that
the symmetric difference of any two {\pema}s is a Hamilton cycle.
Since every conformal cycle is the symmetric difference of
two perfect matchings,
we infer that $J$ has precisely three conformal cycles,
each of which is hamiltonian; these are $M_1 \Delta M_2$,
$M_2 \Delta M_3$ and $M_3 \Delta M_1$.
We let $A$ and $B$ denote the color classes of the cycle $C:=M_1 \Delta M_2$.

\smallskip
First suppose that some edge of $M_3$
has one end in $A$ and the other end in $B$.
We let $e:=ab$ denote such an edge, where $a \in A$ and $b \in B$.
We let $P_1$ and $P_2$ denote the two distinct $ab$-paths of $C$.
Each of them is an odd path. Consequently, each of $P_1 \cup \{e\}$
and $P_2 \cup \{e\}$ is a conformal cycle of $J$. As noted earlier, each conformal cycle
is hamiltonian, whence each of $P_1$ and $P_2$ has exactly one edge.
Thus $J$ is $\Theta$.

\smallskip
We now consider the remaining case --- for each edge $e \in M_3$, either $e$
has both ends in~$A$, or otherwise $e$ has both ends in $B$.
Consequently, for each $e \in M_3$, the spanning subgraph $C \cup \{e\}$
comprises two odd cycles. Among all such odd cycles of $J$, we choose one
that is shortest, say $Q$. Thus $Q=P \cup \{e\}$ where $P$ is an even path of $C$
and $e \in M_3$. Adjust notation so that $e:=a_1a_2$ has both ends in $A$.
We let $b_1 \in B$ denote the unique neighbour of $a_1$ in~$P$,
and we let $f:=b_1b_2$ denote the unique edge of $M_3$ incident with~$b_1$.
Note that $b_2 \in B$. Furthermore, by choice of $Q$, the vertex $b_2$ does not
lie on $P$. Consequently, $a_1, b_1, a_2, b_2$ appear in that order on
the Hamilton cycle $C$.
Observe that the spanning subgraph $C \cup \{e,f\}$
has two even cycles --- each of which is distinct from $C$.
Furthermore, each of them
uses both $e$ and $f$, and each of them is a conformal cycle of $J$.
Since each conformal cycle is hamiltonian, we infer that
$J$ has exactly four vertices; whence $J$ is~$K_4$.
\end{proof}

Now, we proceed to prove another result that we will find useful in the next section.

\begin{lem}
\label{lem:edge-set-union-of-three-pms}
Every \mcg, whose edge-set may be expressed as the union of 
three distinct {\pema}s, is a \bisub\ of a $2$-connected cubic graph.
\end{lem}
\begin{proof}
Let $G$ denote a \mcg\ whose edge-set may be expressed a the union of
three distinct {\pema}s.
It follows immediately that each vertex of $G$
has degree two or three. Furthermore, $G$ has at least two cubic vertices.

\begin{sta}
\label{sta:odd-path}
Let $P$ be a maximal path of $G$, whose each internal vertex is of degree~two
in~$G$. Then $P$ is an odd path.
\end{sta}

\begin{proof}
Let $v_1$ and $v_2$ denote the ends of $P$.
By maximality of $P$, each of $v_1$~and~$v_2$ is a cubic vertex of~$G$.
For $i \in \{1,2\}$, let $e_i$ and $f_i$ denote the two edges of $\partial(v_i)-E(P)$.
Observe that, for any \pema~$M$ of~$G$,
the set $M \cap \{e_1,f_1,e_2,f_2\}$ is a singleton.
Consequently, $E(G)$ can not be covered by three {\pema}s,
contrary to our hypothesis.
\end{proof}

Statement~\ref{sta:edge-set-union-of-three-pms}
immediately implies that $G$ is a \bisub\ of some cubic graph.
This proves Lemma~\ref{lem:edge-set-union-of-three-pms}.
\end{proof}

\subsection{A proof of the $\Theta$ - $K_4$ Theorem}
\label{sec:theta-K4}

In this section, we will prove Theorem~\ref{thm:mcg-theta-K4-comi} which states
that every \mcg\ is either $\Theta$-based, or $K_4$-based, or both.
We will find the following easy fact useful.

\begin{lem}
\label{lem:union-of-pms}
Let $M_1, M_2, \dots, M_r$ denote distinct perfect matchings of a (matchable) graph~$G$,
and let $H$ denote the (spanning) subgraph formed by the edge-set
$M_1 \cup M_2 \cup \cdots \cup M_r$. Then each component of $H$ is a matching
covered conformal subgraph of $G$. \qed
\end{lem}

\smallskip
By a {\it \claw} of a graph~$G$,
we mean any subgraph that is formed by three edges
that are all incident with a common vertex.
(We admit the possibility that any two, or perhaps all three, are parallel edges.)
Ageev, Benchetrit, Seb{\H{o}} and Szigeti~\cite{abss11} proved that in a \mcg~$G$,
each \claw\ participates in a subgraph~$H$ (of~$G$) that is either a \bisub\ of $\Theta$ or of $K_4$.
We prove a slightly stronger statement that immediately implies Theorem~\ref{thm:mcg-theta-K4-comi}.

\begin{thm}
\label{thm:claw-theta-K4-comi}
In a \mcg~$G$, each \claw\ is a subgraph of a conformal subgraph~$H$ (of~$G$) that is either a
\bisub\ of $\Theta$ or of $K_4$.
\end{thm}

\begin{proof}
Suppose that the statement is false.
Among all counterexamples, let $G$ denote a \mcg, and let $K$ denote a \claw\ of $G$,
so that $|E(G)|$ is minimum.

\begin{sta}
\label{sta:degree-two-vertices-comprise-stable-set}
The degree two vertices of $G$ comprise a stable set.
\end{sta}

\begin{proof}
Suppose not; let $v_1$ and $v_2$ denote two adjacent vertices, each of degree two.
Let $u_1$ denote the neighbor of $v_1$ distinct from $v_2$, and let $u_2$ denote the neighbor of $v_2$
distinct from $v_1$. Observe that $G-v_1-v_2+u_1u_2$ is a counterexample with fewer edges (by modifying
the \claw\ $K$ appropriately), contrary to our choice of $G$.
\end{proof}

We let $e_1,e_2$ and $e_3$ denote the edges of the \claw~$K$.

\begin{sta}
\label{sta:edge-set-union-of-three-pms}
For each $i \in \{1,2,3\}$, let $M_i$ denote any perfect of $G$ such that $e_i \in M_i$.
Then $E(G) = M_1 \cup M_2 \cup M_3$.
\end{sta}

\begin{proof}
By Lemma~\ref{lem:union-of-pms},
each component of the subgraph formed by the edge-set $M_1 \cup M_2 \cup M_3$
is a matching covered conformal subgraph of $G$; the \claw\ $K$ is a subgraph
of one such component, say $G'$. If $G' \neq G$ then $G'$ is a counterexample with
$|E(G')| < |E(G)|$,
contrary to our choice of~$G$.
Thus $G' = G$, whence $E(G) = M_1 \cup M_2 \cup M_3$.
\end{proof}

Consequently,
by Lemma~\ref{lem:edge-set-union-of-three-pms},
$G$ is a \bisub\ of a cubic graph.
Furthermore, Statement~\ref{sta:degree-two-vertices-comprise-stable-set}
implies that $G$ is in fact a cubic graph.

\smallskip
For each $i \in \{1,2,3\}$, we let $M_i$ denote a fixed perfect matching of $G$
such that $e_i \in M_i$.

\begin{sta}
\label{sta:uniqueness-of-pms}
For each $i \in \{1,2,3\}$,
$M_i$ is the unique perfect matching of $G$ that contains $e_i$.
\end{sta}

\begin{proof}
Let $M'_1$ denote any perfect matching of $G$ that contains $e_1$.
By statement~\ref{sta:edge-set-union-of-three-pms},
$E(G) = M_1 \cup M_2 \cup M_3 = M'_1 \cup M_2 \cup M_3$.
Since $G$ is cubic, this immediately implies that $M_1 = M'_1$.
\end{proof}

It follows from statement~\ref{sta:uniqueness-of-pms}
that $G$ has exactly three perfect matchings.
By Theorem~\ref{thm:mcg-3-pema},
the cubic graph~$G$ is either $\Theta$ or $K_4$; clearly, neither of these
is a counterexample, contrary to our assumption.
This completes the proof of Theorem~\ref{thm:claw-theta-K4-comi}.
\end{proof}

\begin{figure}[!htb]
\centering
\begin{tikzpicture}

\draw (-1,1) to [out=90,in=180] (2,3.5) to [out=0,in=90] (5,1);

\draw (0,0) -- (4,0) -- (4,2) -- (0,2) -- (0,0);
\draw (2,0) -- (2,2);
\draw (0,0) -- (-1,1) -- (0,2);
\draw (4,0) -- (5,1) -- (4,2);

\draw (0,0)node{}node[below,nodelabel]{$y$};
\draw (4,0)node{}node[below,nodelabel]{$z$};
\draw (4,2)node{}node[above,nodelabel]{$x$};
\draw (0,2)node{}node[above,nodelabel]{$w$};
\draw (-1,1)node{}node[left,nodelabel]{$s$};
\draw (5,1)node{}node[right,nodelabel]{$t$};
\draw (2,0)node{}node[below,nodelabel]{$v$};
\draw (2,2)node{}node[above,nodelabel]{$u$};

\end{tikzpicture}
\caption{The Bicorn $R_8$ has all three types of vertices}
\label{fig:Bicorn}
\end{figure}
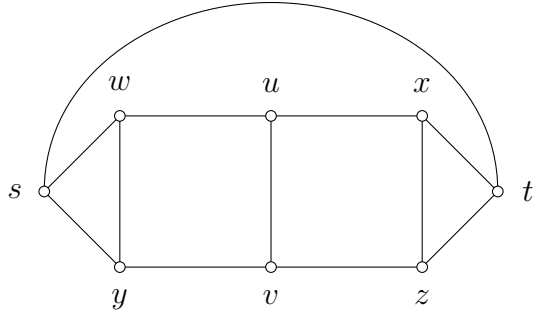

Given a vertex~$v$ of a $2$-connected cubic graph~$G$, let $K^v$ denote the \claw\ formed by the three edges incident at $v$.
It follows from Theorem~\ref{thm:claw-theta-K4-comi} that (i) either $K^v$ is a subgraph of a conformal subgraph~$H$
(of~$G$) that is a \bisub\ of~$\Theta$, or (ii) $K^v$ is a subgraph of a conformal subgraph~$H$ (of~$G$)
that is a \bisub\ of~$K_4$, or (iii) both possibilities hold.
Based on this observation, we may partition the vertex set of any $2$-connected cubic graph~$G$ into three (possibly empty) sets.
We give an example below.

We use ${\sf Aut}(G)$ to denote the automorphism group of a graph~$G$.
The well-known graph shown in Figure~\ref{fig:Bicorn} is the {\em Bicorn}, denoted by $R_8$.
Observe that ${\sf Aut}(R_8)$ has precisely three orbits: $\{s,t\}, \{u,v\}$ and $\{w,x,y,z\}$.
The reader may verify that the partition of $V(R_8)$, as described in the preceding paragraph,
is the following: (i) $\{u,v\}$, (ii) $\{s,t\}$, and (iii)~$\{w,x,y,z\}$.

\subsection{A proof of the $K_4$-decoration Theorem}
\label{sec:K4-decoration-thm-proof}

The following is easily observed by considering the symmetric difference of two appropriately chosen {\pema}s.

\begin{lem}
\label{lem:conformal-cycle-containing-adjacent-edges}
In a \mcg, there exists a conformal cycle containing any two adjacent edges. \qed
\end{lem}

In fact, Little~\cite{litt74} proved the above statement for any two (not necessarily adjacent) edges. However,
we won't need that strengthening for the following easy consequence.

\begin{cor}
\label{cor:splicing-at-good-vertex}
Let $J$ be any \mcg\ distinct from $K_2$,
and let $C:=\partial(X)$ denote a separating cut of a \mcg~$G$.
Assume that
the $C$-contraction $G_1:=G/(X \rightarrow x)$ has a conformal subgraph,
say~$H_1$, that is a bi-subdivision of~$J$
so that either $x \notin V(H_1)$ or $x$ is a vertex of $H_1$ of degree two.
Then $G$ has a conformal subgraph $H$ that is a bi-subdivision of~$J$.
\end{cor}
\begin{proof}
Let $M_1$ denote a \pema\ of $G_1-V(H_1)$.

If $x \notin V(H_1)$, let $e$ denote the unique edge in $M_1 \cap \partial(x)$.
Now, let $H:=H_1$, and let $M:=M_1 \cup M_e$ where $M_e$ is a \pema\ of~$G/{\overline{X}}$ containing $e$.
Observe that $M$ is a \pema\ of $G-V(H)$.

Now suppose that $x$ is a vertex of~$H_1$ of degree two,
and let $e$~and~$f$ denote the two edges in $E(H_1) \cap \partial(x)$.
By Lemma~\ref{lem:conformal-cycle-containing-adjacent-edges},
$G/{\overline{X}}$ has a conformal cycle, say~$Q$,
containing $e$~and~$f$; we let $M_Q$ denote a \pema\ of~$(G/{\overline{X}}) - V(Q)$. Observe that
$H:=H_1 \cup Q$ is a bi-subdivision of $J$, and that $M:=M_1 \cup M_Q$ is a \pema\ of $G-V(H)$.

In both cases, $H$ is a conformal subgraph (of~$G$) that is a bi-subdivison of~$J$.
\end{proof}

Now, we are ready to prove the \KFdec\ Theorem (\ref{thm:K4-decoration}) stated below
in the contrapositive.

\begin{thm}
\label{thm:K4-decoration-contrapositive}
      {\sc [$K_4$-decoration Theorem]}
\breaksta
Let $J$ denote any $2$-connected cubic graph,
and let $C$ denote a separating cut of a \mcg~$G$.
If either \mbox{$C$-contraction} of~$G$ is $K_4(J)$-based
then $G$ is $K_4(J)$-based.
\end{thm}
\begin{proof}
We let $C:=\partial(X)$, and we let $G_1:=G/(X \rightarrow x)$ and $G_2:=G/(\overline{X} \rightarrow \overline{x})$
denote the \mbox{$C$-contractions} of~$G$. Assume that $G_1$ is $K_4(J)$-based; that is, $G_1$
has a conformal subgraph~$H_1$ that is a \bisub\ of some \KFdec, say~$L$, of~$J$.

If either $x \notin V(H_1)$ or $x$ is a vertex of $H_1$ of degree two then,
by Corollary~\ref{cor:splicing-at-good-vertex},
$G$~has a conformal subgraph~$H$ that is a \bisub\ of~$L$. The desired conclusion holds.
Henceforth, we consider that $x$ is a cubic vertex of~$H_1$.

Let $e_1,e_2$ and $e_3$ denote the three edges of~$H_1$ incident at $x$;
these edges form a claw, say~$K$, in $G_2$.
By Theorem~\ref{thm:claw-theta-K4-comi}, $K$ is a subgraph of a conformal subgraph (of~$G_2$), say~$H_2$,
that is either a \bisub\ of $\Theta$ or of~$K_4$.
For $i \in \{1,2\}$, we let $M_i$ denote a \pema\ of $G_i - V(H_i)$.

If $H_2$ is a \bisub\ of~$\Theta$, observe that $H:=H_1 \cup H_2$ is a \bisub\ of~$L$,
and that $M:=M_1 \cup M_2$ is a \pema\ of~$G-V(H)$.
The desired conclusion holds. Henceforth, we consider that $H_2$ is a \bisub\ of~$K_4$.

We let $L:=J^T \odot K_4$ for some fixed $T \subseteq V(J)$. Since $H_1$ is a bi-subdivision of~$L$,
we may fix a correspondence between the cubic vertices of~$H_1$ and the vertices of~$L$;
in fact, we may use the same labels to refer to the vertices of~$L$.
As per this, either $x$ is an original vertex of $L$ (i.e., $x \notin T$), or otherwise $x$
is a decoration vertex of~$L$. We consider these cases separately.

If $x$ is an original vertex of~$L$ then $x$ may also be regarded to be a vertex of~$J$.
Observe that $H:=H_1 \cup H_2$ is a \bisub\
of~$L':=J^{T'} \odot K_4$ where $T':=T \cup \{x\}$, and that $M:=M_1 \cup M_2$
is a \pema\ of~$G-V(H)$. The desired conclusion holds.

\begin{figure}[!htb]
\centering
\begin{tikzpicture}[scale=1.5]

\draw (0,0) -- (2,0) -- (2,2) -- (0,2) -- (0,0);

\draw (3,1) -- (1.5,1);
\draw (3,1) -- (1.5,0.5);
\draw (3,1) -- (1.5,1.5);

\draw (1.5,1)node{};
\draw (1.5,0.5)node{};
\draw (1.5,1.5)node{};

\draw (3,1)node{};
\draw (3,1.2)node[nodelabel]{$v$};

\draw (1.5,-1)node[nodelabel]{$J$};

\end{tikzpicture}
\hspace*{0.3in}
\begin{tikzpicture}[scale=1.5]

\draw (0,0) -- (2,0) -- (2,2) -- (0,2) -- (0,0);
\draw[dashed] (3,1) -- (3.5,0.5) -- (3.5,1.5) -- (3,1);

\draw[dashed] (3.5,1.5) -- (1.5,1.5);
\draw[dashed] (3.5,0.5) -- (1.5,0.5);
\draw[dashed] (3,1) -- (1.5,1);

\draw (1.5,1)node{};
\draw (1.5,0.5)node{};
\draw (1.5,1.5)node{};

\draw (3,1)node{};
\draw (3,1.2)node[nodelabel]{$y$};
\draw (3.5,0.5)node{};
\draw (3.5,0.3)node[nodelabel]{$x$};
\draw (3.5,1.5)node{};
\draw (3.5,1.7)node[nodelabel]{$z$};

\draw (1.75,-0.8)node[nodelabel]{$H_1$};

\draw[dashed] (5,1) -- (6.5,1.5);
\draw[dashed] (5,1) -- (6.5,0.5);
\draw[dashed] (5,1) -- (6.5,1);
\draw[dashed] (6.5,0.5) -- (6.5,1.5);
\draw[dashed] (6.5,0.5) to [out=0,in=270] (7,1) to [out=90,in=0] (6.5,1.5);

\draw (5,1)node{};
\draw (4.8,1)node[nodelabel]{$\overline{x}$};

\draw (6.5,1.5)node{};
\draw (6.5,1.7)node[nodelabel]{$z'$};
\draw (6.5,0.5)node{};
\draw (6.7,1)node[nodelabel]{$y'$};
\draw (6.5,1)node{};

\draw (6,-0.8)node[nodelabel]{$H_2$};

\end{tikzpicture}

\begin{tikzpicture}[scale=1.5]

\draw (0,0) -- (2,0) -- (2,2) -- (0,2) -- (0,0);
\draw[dashed] (3,1) -- (3.5,1.5);

\draw[ultra thick,dashed] (3.5,1.5) -- (1.5,1.5);
\draw[ultra thick,dashed] (6.5,0.5) -- (1.5,0.5);
\draw[ultra thick,dashed] (3,1) -- (1.5,1);
\draw[ultra thick,dashed] (3.5,1.5) -- (6.5,1.5);
\draw[ultra thick,dashed] (3,1) -- (6.5,1);

\draw (1.5,1)node{};
\draw (1.5,0.5)node{};
\draw (1.5,1.5)node{};

\draw (3,1)node{};
\draw (3,1.2)node[nodelabel]{$y$};
\draw (3.5,1.5)node{};
\draw (3.5,1.7)node[nodelabel]{$z$};

\draw[ultra thick,dashed] (6.5,0.5) -- (6.5,1);
\draw[dashed] (6.5,1) -- (6.5,1.5);
\draw[ultra thick,dashed] (6.5,0.5) to [out=0,in=270] (7,1) to [out=90,in=0] (6.5,1.5);


\draw (6.5,1.5)node{};
\draw (6.5,1.7)node[nodelabel]{$z'$};
\draw (6.5,0.5)node{};
\draw (6.7,1)node[nodelabel]{$y'$};
\draw (6.5,1)node{};

\draw (3.5,-0.8)node[nodelabel]{$H_1 \cup H_2$};

\end{tikzpicture}

\caption{Illustration for the proof of the \KFdec\ Theorem (\ref{thm:K4-decoration-contrapositive})}
\label{fig:K4-decoration-contrapositive-last-case}
\end{figure}
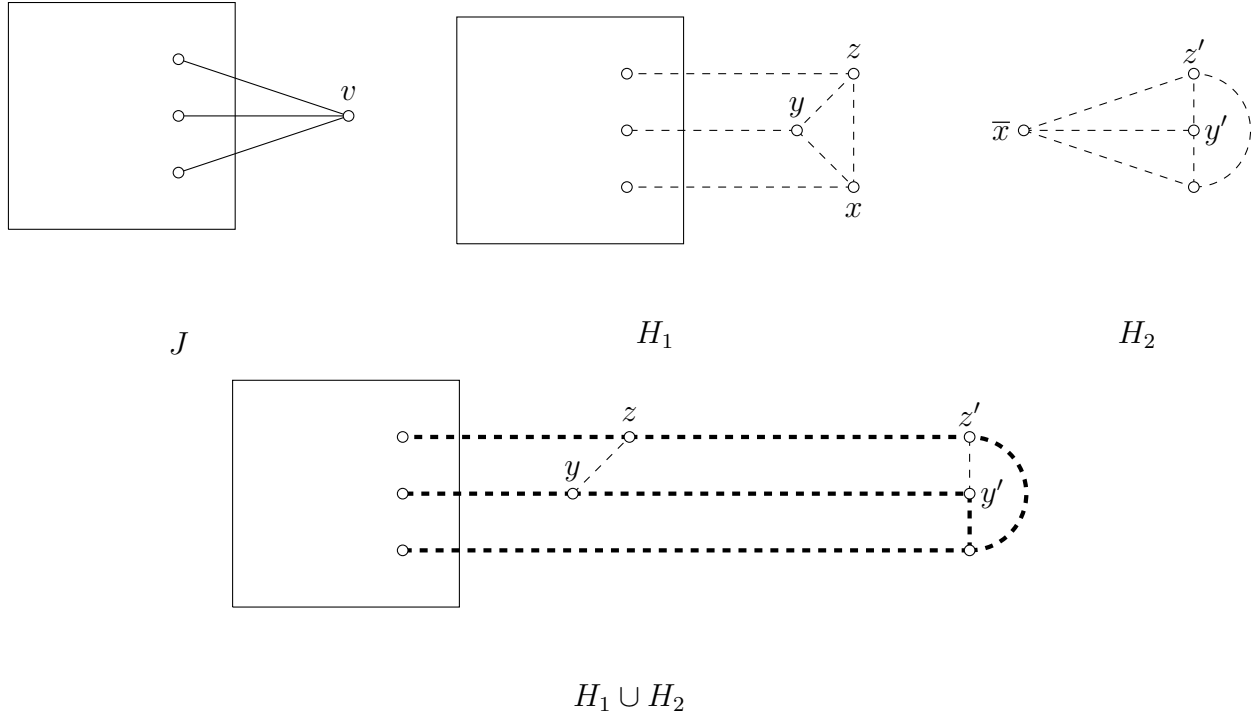

Finally, suppose that $x$ is a decoration vertex of~$L$. As discussed earlier,
we may fix a three-to-one correspondence between the decoration vertices of~$L$ and the set~$T$,
say \mbox{$\tau:V(L) - \overline{T} \rightarrow T$}.
Let $v:=\tau(x)$, and let $y$ and $z$ denote the other two
decoration vertices of $L$ that also correspond to $v$.
See Figure~\ref{fig:K4-decoration-contrapositive-last-case};
each dashed line joining two vertices represents a path of odd length.
We let $P_{yz}$ denote the unique $yz$-path in~$H_1$
whose each internal vertex is of degree two (in~$H_1$); the paths $P_{xy}$ and $P_{xz}$ are defined analogously.
Adjust notation so that $e_2 \in E(P_{xy})$ and $e_3 \in E(P_{xz})$.

Note that $H_2$ is a \bisub\ of~$K_4$. In particular, it has three paths that start at $\overline{x}$ ---
say $P_1, P_2$ and $P_3$ --- that are disjoint (except for $\overline{x}$) and end at distinct cubic vertices.
Adjust notation so that $e_2 \in E(P_2)$ and $e_3 \in E(P_3)$,
and let $y'$ and $z'$ denote the ends of $P_2$~and~$P_3$,
respectively, that are distinct from~$\overline{x}$;
see Figure~\ref{fig:K4-decoration-contrapositive-last-case}.
We let $P_{y'z'}$ denote the unique $y'z'$-path in~$H_2$ whose each internal vertex is of degree two (in~$H_2$).
We let $M_{yz}$ denote the \pema\ of~$P_{yz} - y - z$, and we define $M_{y'z'}$ analogously.

Observe that $H:=H_1 \cup H_2 - E(P_{yz}) - E(P_{y'z'})$ is a \bisub\ of $J^{T'} \odot K_4$,
where $T':=T-v$,
and that $M:=M_1 \cup M_2 \cup M_{yz} \cup M_{y'z'}$ is a \pema\ of~$G-V(H)$.
See Figure~\ref{fig:K4-decoration-contrapositive-last-case}; the thick dashed lines
indicate the edges that are retained in $H$.
The desired conclusion holds.

This proves the \KFdec\ Theorem (\ref{thm:K4-decoration-contrapositive}).
\end{proof}

\section{Applications}
\label{sec:applications}

It was shown by Valiant \cite{vali79} that,
in general,
it is not possible to compute the number of {\pema}s of a graph in polynomial-time
(unless {\bf P} $=$ {\bf NP}).
(It is easily observed that one may restrict attention to {\mcg}s.)

An orientation~$D$ of a \mcg~$G$ is a {\em \PO} if each conformal cycle
has an odd number of edges oriented clockwise (and, consequently, an odd number of edges oriented anticlockwise).
The significance of {\PO}s arises from the fact that, whenever a graph~$G$ is known to admit a {\PO},
such an orientation may be computed in polynomial-time, and it may consequently be used to compute the
number of perfect matchings of~$G$ in polynomial-time. However, not all graphs admit a \PO.
For instance, it is easily argued that $K_{3,3}$ does not admit one.

A \mcg~$G$ is {\em \pfaff} if $G$ admits a \PO; otherwise, $G$ is \nonpfaff.
It was shown by Kasteleyn \cite{kast63} that all planar graphs are \pfaff.
A natural question arises: is it possible to characterize
the class of \pfaff\ graphs (or subclasses of interest therein) in terms of `forbidden minors'?
The following seminal result of Little~\cite{litt75} achieved precisely that
for the class of \pfaff\ bipartite graphs.

\begin{thm}
\label{thm:little}
{\sc [Little's Theorem]}
A bipartite \mcg~$G$ is \pfaff\ if and only if
$G$ is \KTTF.
\end{thm}

Thus, Little's Theorem provides a characterization of \pfaff\ bipartite {\mcg}s in terms of forbidden {\comi}s,
and laid the foundation of the quest for more such results.

A few decades later, Fischer and Little~\cite{fili01} characterized \pfaff\ `\NB' graphs in terms
of forbidden {\Smi}s, as well as in terms of forbidden {\comi}s.
More recently, Carvalho, Lucchesi and Murty~\cite{clm12} generalized Little's Theorem (\ref{thm:little})
and characterized \pfaff\ `solid' graphs in terms of forbidden {\Smi}s; in their paper, they also stated
a characterization in terms of forbidden {\comi}s; however, they omitted the proof.
In the next two sections, we use our Main Theorem~(\ref{thm:conformal-minors-vs-S-minors-cubic-long-version})
to deduce the `\comi\ version' from the `\Smi\ version' --- for both of these results.

\subsection{Pfaffian solid graphs}

A \mcg~$G$ is {\em solid} if each of its separating cuts is a tight cut;
otherwise $G$ is {\em nonsolid}.
In particular, a brick is solid if and only if it is devoid of nontrivial separating cuts.
It is easily proved that every bipartite \mcg\ is solid; see \cite[Corollary 1.9]{lckm18}.
Thus solid graphs are a generalization of bipartite graphs.

In order to familiarize ourselves
with the notion, let us figure out which of the ten {\KFdec}s of $K_{3,3}$ are solid,
and which ones are nonsolid;
we will find this useful soon. Of course, $K_{3,3}$ itself is solid.
The following consequence of \cite[Theorem 2.25]{clm02} will come in handy;
see also \cite[Corollary 1.13]{lckm18}.

\begin{thm}
\label{thm:mcg-solid-iff-bricks-solid}
A \mcg\ is solid if and only if each of its bricks is solid.
\end{thm}

It is easily observed that,
for each of the nonbipartite graphs shown in Figure~\ref{fig:solid-K4-decorations-of-K33},
its tight cut decomposition yields (one or more) copies of $K_4$ and (one copy of) $K_{3,3}$.
Clearly, $K_4$ is solid; whence, by Theorem~\ref{thm:mcg-solid-iff-bricks-solid},
each graph shown in Figure~\ref{fig:solid-K4-decorations-of-K33} is solid.

\begin{figure}[!htb]
\centering

\begin{tikzpicture}

\draw (0,0) -- (0,2);
\draw (0,0) -- (4,2);

\draw (2,0.3) -- (1.6,-0.3) -- (2.4,-0.3) -- (2,0.3);
\draw (1.6,-0.3) -- (0,2);
\draw (2.4,-0.3) -- (4,2);

\draw (4,0) -- (0,2);
\draw (4,0) -- (4,2);

\draw (2,1.7) -- (1.6,2.3) -- (2.4,2.3) -- (2,1.7);
\draw (1.6,2.3) -- (0,0);
\draw (2.4,2.3) -- (4,0);

\draw (2,0.3) -- (2,1.7);

\draw (0,0)node{};
\draw (4,0)node{};

\draw (2,0.3)node{};
\draw (1.6,-0.3)node{};
\draw (2.4,-0.3)node{};

\draw (0,2)node{};
\draw (4,2)node{};

\draw (2,1.7)node{};
\draw (1.6,2.3)node{};
\draw (2.4,2.3)node{};

\draw (2,-1.5)node[nodelabel]{$K_{3,3}^{\{a_2,b_2\}} \odot K_4$};

\end{tikzpicture}
\vspace*{-0.4in}
\caption{The smallest \KFdec\ of~$K_{3,3}$ that is a nonsolid brick}
\label{fig:smallest-nonsolid-K4-decoration-of-K33}
\end{figure}

For the remaining six graphs in $K_4(K_{3,3})$, observe that each of them
is a \KFdec\ of the graph~$G:=K_{3,3}^{\{a_2,b_2\}} \odot K_4$
shown in Figure~\ref{fig:smallest-nonsolid-K4-decoration-of-K33}.
One may easily verify that $G$ is a brick --- for instance, by using Theorem~\ref{thm:elp-brick-characterization}.
Furthermore, $G$ is nonsolid since it is obtained by splicing $K_{3,3} \odot K_4$ and $K_4$, and the corresponding
splicing cut is a nontrivial separating cut.
A similar reasoning, combined with the following result~\cite[Corollary 2.8]{clm05},
implies that each of these six graphs is a nonsolid brick.

\begin{prp}
\label{prp:splicing-of-two-cubic-bricks}
The splicing of any two cubic bricks yields a cubic brick.
\end{prp}

We let $\mathcal{S}$ denote the set comprising the four graphs shown in
Figure~\ref{fig:solid-K4-decorations-of-K33}.
Our discussion above proves the following statement.

\begin{prp}
\label{prp:solid-K4-decorations-of-K33}
Among the ten graphs in $K_4(K_{3,3})$, the only ones that are solid are those that belong to $\mathcal{S}$. \qed
\end{prp}

The notion of solid graphs was introduced by Carvalho, Lucchesi and Murty in \cite{clm02};
their significance arises from the fact that they are a generalization of bipartite graphs and they
share some of the nice properties. Consequently, they often admit easier answers to certain problems
that are yet unsolved for all nonbipartite graphs. A case in point is the following
characterization of \pfaff\ solid graphs in terms of forbidden {\Smi}s --- established by the same authors~\cite{clm12}.

\begin{thm}
\label{thm:solid-pfaffian-forbidden-S-minors}
A solid \mcg~$G$ is \pfaff\ if and only if $K_{3,3}$ is not an \Smi\ of~$G$.
\end{thm}

The above theorem may be viewed as a generalization of Little's Theorem
as explained in \cite{clm12}. The authors also stated a characterization of \pfaff\ solid graphs
in terms of forbidden {\comi}s (Theorem~\ref{thm:solid-pfaffian-forbidden-conformal-minors};
however, they omitted its proof (to limit the length).
As an application of our Main Theorem (\ref{thm:conformal-minors-vs-S-minors-cubic-long-version}),
we will deduce Theorem~\ref{thm:solid-pfaffian-forbidden-conformal-minors}
from Theorem~\ref{thm:solid-pfaffian-forbidden-S-minors}. In order to do so,
we just need one more ingredient stated below; see \cite[Corollary 1.38]{lckm18}.

\begin{thm}
\label{thm:conformal-minors-solid}
Every conformal minor of a solid \mcg\ is also solid.
\end{thm}

We are now ready to deduce (from Theorem~\ref{thm:solid-pfaffian-forbidden-S-minors})
the characterization of \pfaff\ solid graphs in terms of
forbidden {\comi}s due to Carvalho, Lucchesi and Murty~\cite{clm12}.

\begin{thm}
\label{thm:solid-pfaffian-forbidden-conformal-minors}
A solid \mcg~$G$ is \pfaff\ if and only if $G$ is $\mathcal{S}$-free.
\end{thm}
\begin{proof}
It follows immediately from Theorem~\ref{thm:solid-pfaffian-forbidden-S-minors} and
our Main Theorem (\ref{thm:conformal-minors-vs-S-minors-cubic-long-version})
that a solid \mcg~$G$ is \pfaff\ if and only if $G$ is $K_4(K_{3,3})$-free.
By invoking Theorem~\ref{thm:conformal-minors-solid} and
Proposition~\ref{prp:solid-K4-decorations-of-K33},
we conclude that a solid \mcg~$G$ is \pfaff\ if and only if $G$ is $\mathcal{S}$-free.
\end{proof}

\subsection{Pfaffian \NB\ graphs}

A \mcg~$G$ is {\em \NB} if it has a removable double ear, say $(P,Q)$,
such that \mbox{$G-P-Q$} is bipartite and \mc.
To put it differently, a \mcg~$G$ is near-bipartite if and only if
$G$ admits an ear decomposition $(G_1, G_2, \dots, G_r)$ such that, for each $2 \leq i \leq r-1$, 
the graph $G_{i-1}$ is obtained from $G_i$ by deleting a removable single ear, whereas
$G_{r-1}$ is obtained from $G_r$ by deleting a removable double ear.
It follows that all of the subgraphs $G_1, G_2, \dots, G_{r-1}$ are bipartite;
whereas, the final graph $G:=G_r$ is nonbipartite.
(We shall abbreviate `\NB\ \mcg' to `\NB\ graph'.)

Note that a nonbipartite \mcg~$G$, with minimum degree three or more, is \NB\
if and only if it has a pair of edges $R:=\{\alpha,\beta\}$ such that $G-R$ is bipartite and \mc.
We refer to such a pair~$R$ as a {\em removable doubleton}.
It is easily observed $K_4$ and $\overline{C_6}$ are \NB\
bricks; each of them has precisely three removable doubletons.
With respect to a fixed removable doubleton~$R:=\{\alpha,\beta\}$, we use $A$ and $B$ to denote the
color classes of the {\em underlying bipartite graph} $G-R$ with $\alpha:=a_1a_2$ having both ends in
the white color class $A$,
and $\beta:=b_1b_2$ having both ends in the black color class $B$.

Figure~\ref{fig:Cubeplex-and-Twinplex} shows two cubic \NB\ bricks of order $12$,
denoted by $\Gamma_1$~and~$\Gamma_2$ respectively, that will play
an important role in this section; the names Cubeplex (for $\Gamma_1$) and Twinplex (for $\Gamma_2$)
are due to Norine and Thomas \cite{noth08}.
Each of them has a unique removable doubleton --- as indicated in the figures.
The following is easily verified using the figures provided.

\begin{prp}
\label{prp:automorphisms-of-Cubeplex-and-Twinplex}
$\{a_1\}, \{a_2\}$ and $\{b_1,b_2\}$ are orbits of ${\sf Aut}(\Gamma_1)$,
and $\{a_1,a_2,b_1,b_2\}$ is an orbit of ${\sf Aut}(\Gamma_2)$. \qed
\end{prp}

\begin{figure}[!htb]
\centering
\begin{tikzpicture}[scale=1.2]

\draw (0:0) to [out=135,in=45] (180:1.75);
\draw (0:0) to [out=45, in=135] (0:1.75);
\draw (0:0) to [out=225,in=135] (270:1.75);
\draw (250:1.4)node[nodelabel]{$\alpha$};

\draw (90:1) -- (180:1) -- (270:1) -- (0:1) -- (90:1);
\draw (90:2.5) -- (180:2.5) -- (270:2.5) -- (0:2.5) -- (90:2.5);

\draw (90:1) -- (90:2.5);
\draw (83:1.7)node[nodelabel]{$\beta$};
\draw (180:1) -- (180:1.75) -- (180:2.5);
\draw (270:1) -- (270:1.75) -- (270:2.5);
\draw (0:1) -- (0:1.75) -- (0:2.5);

\draw (0:0)node{};
\draw (330:0.3)node[nodelabel]{$a_1$};

\draw (90:1)node[fill=black]{};
\draw (90:0.7)node[nodelabel]{$b_2$};
\draw (180:1)node{};
\draw (270:1)node[fill=black]{};
\draw (0:1)node{};

\draw (180:1.75)node[fill=black]{};
\draw (270:1.75)node{};
\draw (280:1.75)node[nodelabel]{$a_2$};
\draw (0:1.75)node[fill=black]{};

\draw (90:2.5)node[fill=black]{};
\draw (90:2.8)node[nodelabel]{$b_1$};
\draw (180:2.5)node{};
\draw (270:2.5)node[fill=black]{};
\draw (0:2.5)node{};

\draw (270:3.3)node[nodelabel]{$\Gamma_1$};

\end{tikzpicture}
\hspace*{1in}
\begin{tikzpicture}[scale=1.2]

\draw (80:3.2) -- (40:3.5);
\draw (60:3.4)node[nodelabel]{$\beta$};

\draw (67.5:2) -- (40:3.5);
\draw (247.5:2) to [out=300,in=270] (0:2.8) to [out=90,in=270] (40:3.5);
\draw (40:3.5)node[fill=black]{};
\draw (40:3.8)node[nodelabel]{$b_1$};

\draw (157.5:2) to [out=90,in=180] (80:3.2);
\draw (337.5:2) to [out=45,in=330] (80:3.2);
\draw (80:3.2)node[fill=black]{};
\draw (80:3.5)node[nodelabel]{$b_2$};

\draw (22.5:2) -- (202.5:2);
\draw (112.5:2) -- (292.5:2);
\draw (22.5:1) -- (112.5:1);
\draw (67.5:0.9)node[nodelabel]{$\alpha$};

\draw (22.5:1)node{};
\draw (10:1)node[nodelabel]{$a_1$};
\draw (112.5:1)node{};
\draw (125:1)node[nodelabel]{$a_2$};

\draw (22.5:2) -- (67.5:2) -- (112.5:2) -- (157.5:2) -- (202.5:2) -- (247.5:2) -- (292.5:2) -- (337.5:2) -- (22.5:2);

\draw (22.5:2)node[fill=black]{};
\draw (67.5:2)node{};
\draw (112.5:2)node[fill=black]{};
\draw (157.5:2)node{};
\draw (202.5:2)node[fill=black]{};
\draw (247.5:2)node{};
\draw (292.5:2)node[fill=black]{};
\draw (337.5:2)node{};

\draw (270:3.3)node[nodelabel]{$\Gamma_2$};

\end{tikzpicture}
\caption{Cubeplex ($\Gamma_1$) and Twinplex ($\Gamma_2$)}
\label{fig:Cubeplex-and-Twinplex}
\end{figure}
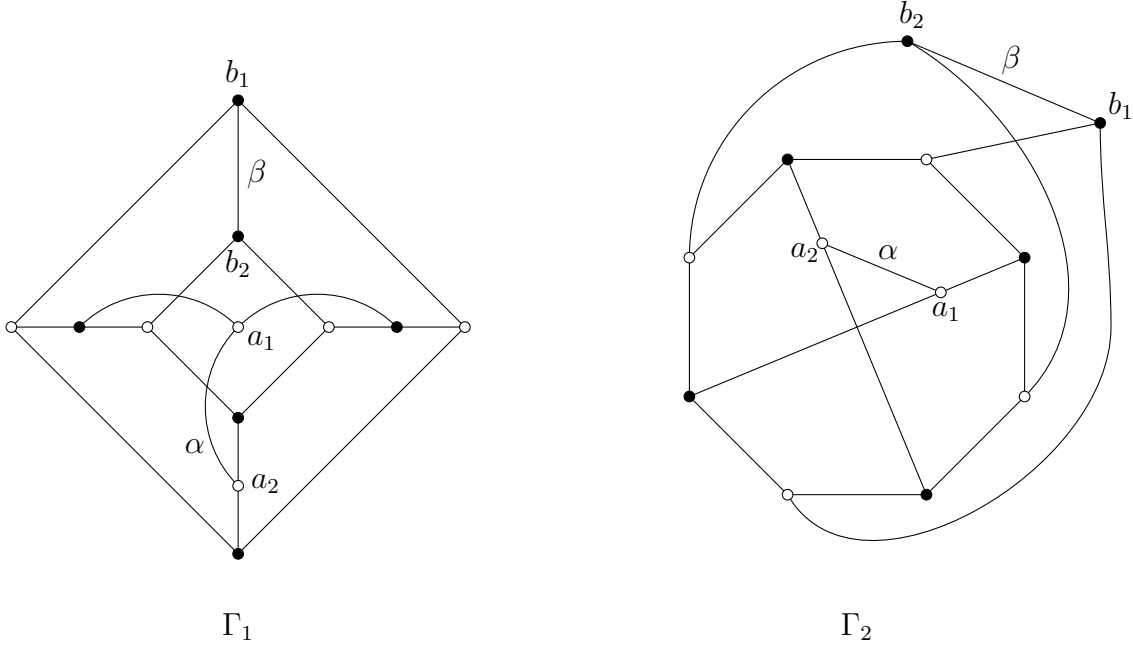

Unlike solid graphs, \NB\ graphs are not a generalization of bipartite graphs;
however, they are very close to being bipartite.
It is for this reason that certain problems are that are open for general nonbipartite graphs
admit easier solutions for the class of \NB\ graphs. For instance, Fischer and Little \cite{fili01}
established the following characterization of \pfaff\ \NB\ graphs in terms
of forbidden {\Smi}s.

\begin{thm}
\label{thm:near-bipartite-pfaffian-forbidden-S-minors}
A \NB\ graph~$G$ is \pfaff\ if and only if (at least) one of the graphs in $\{K_{3,3}, \Gamma_1, \Gamma_2\}$
is an \Smi\ of~$G$.
\end{thm}

As one may expect,
the authors also established a characterization of \pfaff\ \NB\ graphs in terms of forbidden {\comi}s.
As an application of our Main Theorem (\ref{thm:conformal-minors-vs-S-minors-cubic-long-version}),
we will deduce this from Theorem~\ref{thm:near-bipartite-pfaffian-forbidden-S-minors} --- just like
we deduced Theorem~\ref{thm:solid-pfaffian-forbidden-conformal-minors} from 
Theorem~\ref{thm:solid-pfaffian-forbidden-S-minors} in the preceding section.

To this end,
we will need a result analogous to Theorem~\ref{thm:conformal-minors-solid} for \NB\ graphs.
Unfortunately, in general, a \NB\ graph may have conformal minors that are neither bipartite nor \NB.
For instance,
the \NB\ graph $G$ shown in Figure~\ref{fig:near-bipartite-conformal-minor-example}a has a conformal
minor $G-e$ (shown in Figure~\ref{fig:near-bipartite-conformal-minor-example}b) that is neither bipartite nor \NB;
note that $G-e$ is not \mbox{$3$-connected}.
We will show that every \mbox{$3$-connected} cubic conformal minor of a \NB\ graph is either bipartite or \NB.

\vfill

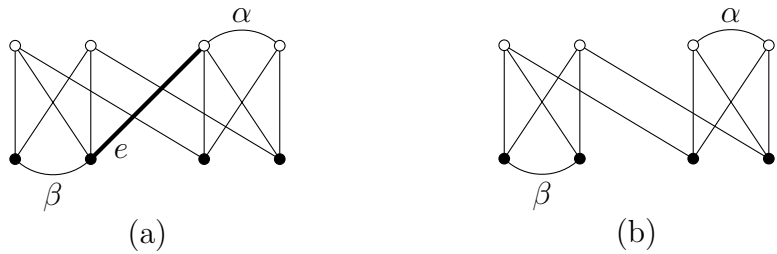
\begin{figure}[!htb]
\centering
\begin{tikzpicture}

\draw (3,1.9)node[nodelabel]{$\alpha$};
\draw (0.5,-0.5)node[nodelabel]{$\beta$};

\draw (0,0) to [out=315,in=225] (1,0);
\draw (2.5,1.5) to [out=45,in=135] (3.5,1.5);

\draw (0,0) -- (0,1.5) -- (1,0) -- (1,1.5) -- (0,0);
\draw (2.5,0) -- (2.5,1.5) -- (3.5,0) -- (3.5,1.5) -- (2.5,0);

\draw[ultra thick] (1,0) -- (2.5,1.5);
\draw (1.4,0.1)node[nodelabel]{$e$};

\draw (0,1.5) -- (2.5,0);
\draw (1,1.5) -- (3.5,0);

\draw (0,0)node[fill=black]{};
\draw (1,0)node[fill=black]{};
\draw (2.5,0)node[fill=black]{};
\draw (3.5,0)node[fill=black]{};

\draw (0,1.5)node{};
\draw (1,1.5)node{};
\draw (2.5,1.5)node{};
\draw (3.5,1.5)node{};

\draw (1.75,-1)node[nodelabel]{(a)};
\end{tikzpicture}
\hspace*{1in}
\begin{tikzpicture}
\draw (3,1.9)node[nodelabel]{$\alpha$};
\draw (0.5,-0.5)node[nodelabel]{$\beta$};

\draw (0,0) to [out=315,in=225] (1,0);
\draw (2.5,1.5) to [out=45,in=135] (3.5,1.5);

\draw (0,0) -- (0,1.5) -- (1,0) -- (1,1.5) -- (0,0);
\draw (2.5,0) -- (2.5,1.5) -- (3.5,0) -- (3.5,1.5) -- (2.5,0);

\draw (0,1.5) -- (2.5,0);
\draw (1,1.5) -- (3.5,0);

\draw (0,0)node[fill=black]{};
\draw (1,0)node[fill=black]{};
\draw (2.5,0)node[fill=black]{};
\draw (3.5,0)node[fill=black]{};

\draw (0,1.5)node{};
\draw (1,1.5)node{};
\draw (2.5,1.5)node{};
\draw (3.5,1.5)node{};

\draw (1.75,-1)node[nodelabel]{(b)};
\end{tikzpicture}
\caption{An example of a \NB\ graph and its conformal minor}
\label{fig:near-bipartite-conformal-minor-example}
\end{figure}

The {\em bipartite index} of a graph~$G$, denoted by~${\sf bi}(G)$,
is the minimum number of edges whose deletion results
in a bipartite graph; in particular, ${\sf bi(G)}=0$ if and only if $G$ is bipartite.
It is easy to see that a \NB\ graph has bipartite index two.
However, not every \mcg\ with bipartite index two is \NB;
see Figure~\ref{fig:near-bipartite-conformal-minor-example}b.

The firt part of the following is immediate. The second part follows from the observation
that bi-subdivisions, unlike general subdivisions, preserve the parities of cycles.

\begin{lem}
\label{lem:bipartite-index-properties}
Let $G$ denote any graph. For each subgraph~$H$ of $G$: ${\sf bi}(H) \leq {\sf bi}(G)$.
If $J$ is a bi-subdivision of $G$, then ${\sf bi}(G) = {\sf bi}(J)$. \qed
\end{lem}

We will find the following lemma useful; it is easily proved using the well-known Hall's Theorem.

\begin{lem}
\label{lem:bipartite-matchable-not-mc}
Let $H[A,B]$ denote a bipartite matchable graph.
If $H$ is not \mc\ then there exist partitions $(A_1,A_2)$ of~$A$ and $(B_1,B_2)$ of~$B$
such that $|A_i| = |B_i|$ for $i \in \{1,2\}$, and there are no edges with one end in~$B_1$
and the other end in~$A_2$. \qed
\end{lem}

We now use the above and simple counting arguments to prove the following.

\begin{prp}
\label{prp:3conn-cubic-mcg-bip-index-at-most-two}
Let $J$ be a $3$-connected cubic graph such that \mbox{${\sf bi}(J) \leq 2$}.
The precisely one of the following holds:
\begin{enumerate}[(i)]
\item either ${\sf bi}(J) =0$ (that is, $J$ is bipartite), or
\item otherwise ${\sf bi}(J)=2$; furthermore, if $D$ is any pair of edges such that $J-D$ is bipartite
then $D$ is a removable doubleton of~$J$; in particular, $J$ is \NB.
\end{enumerate} 
\end{prp}
\begin{proof}
Let $J$ denote a $2$-connected cubic graph.
If $J$ has bipartite index zero, then $J$ is bipartite, and we are done.
Now suppose that $J$ has bipartite index one or two, and let $D$ denote an edge-set of minimum cardinality
such that $J-D$ is bipartite.
Observe that, by choice of $D$, the graph $J-D$ is connected. We let $A$~and~$B$ denote the color classes of
$J-D$. By choice of~$D$, each edge in~$D$ either has both ends in~$A$, or otherwise in~$B$.

First suppose that one of $A$~and~$B$, say~$A$, is a stable set in~$J$; consequently,
each edge in~$D$ has both ends in~$B$. Since $J$ is cubic, $|\partial_J(A)| = 3|A|$.
A simple counting argument shows that $|\partial_J(B)| = 3|B| - 2|D|$;
which is either $3|B|-2$ or $3|B|-4$.
However, $|\partial_J(A)| = |\partial_J(B)|$;
this leads us to an arithmetic contradiction.

Thus, neither $A$ nor $B$ is a stable set in~$J$. In particular, $J$ has bipartite index two, and $|D|=2$.
We let $D:=\{\alpha,\beta\}$, and adjust notation so that $\alpha$ has both ends in~$A$; whence $\beta$
has both ends in~$B$.
Let us first consider the case in which $J$ has no \pema\ containing both $\alpha$~and~$\beta$, and let
$M_\alpha$ denote a \pema\ containing $\alpha$; since $\beta \notin M_\alpha$, we infer that $|A| = |B|+2$.
Using an analogous argument, $|B| = |A|+2$. Once again, an arithmetic contradiction.
Ergo, $J$ has a \pema\ containing both $\alpha$ and $\beta$; this implies that $|A| = |B|$; furthermore,
for each \pema~$M$ of~$J$, either $D \subseteq M$, or $D \cap M = \emptyset$. Observe
that $J-D$ is matchable; to see this, consider any \pema\ containing an edge incident with~$\alpha$. It remains
to argue that $J-D$ is \mc.

Suppose not. By Lemma~\ref{lem:bipartite-matchable-not-mc},
there exist partitions $(A_1, A_2)$ of~$A$ and $(B_1,B_2)$ of~$B$ such that
$|A_i| = |B_i|$ for $i \in \{1,2\}$, and there are no edges with one end in $B_1$ and the other end in $A_2$.
We let $X:=A_1 \cup B_1$ and $C:=\partial(X)$. Since $J$ is cubic, $|\partial_J(X)|$ has the same partiy
as~$|X|$, and since $J$ is $3$-connected, $|\partial_J(X)| \geq 4$. We thus infer that
there are at least two edges in $J$ that have one end in $A_1$ and the other end in $B_2$.
Since $J$ is cubic, we use a simple counting argument to infer the following:
(i) $\beta$ has both ends in~$B_1$, and likewise, $\alpha$ has both ends in~$A_2$, and (ii) $|\partial_J(X)| = 2$.
Ergo, we have a contradiction.

In summary, if $J$ has bipartite index one or two, then $J$ is in fact near-bipartite (and
has bipartite index two). This completes the proof of Proposition~\ref{prp:3conn-cubic-mcg-bip-index-at-most-two}.
\end{proof}

This leads us to the desired result applicable
to \NB\ graphs (that is similar to Theorem~\ref{thm:conformal-minors-solid} which applied to solid graphs).

\begin{cor}
\label{cor:3-connected-cubic-conformal-minors-near-bipartite}
Every $3$-connected cubic conformal minor of a \NB\ graph is either bipartite or \NB.
\end{cor}
\begin{proof}
Let $G$ denote a \NB\ graph that has a conformal minor~$J$ that is $3$-connected and cubic.
In particular, $G$ has a subgraph~$H$ that is a bi-subdivision of~$J$.
Since $G$ is \NB, ${\sf bi}(G)=2$.
By Lemma~\ref{lem:bipartite-index-properties}, ${\sf bi}(J) = {\sf bi}(H) \leq 2$.
By Proposition~\ref{prp:3conn-cubic-mcg-bip-index-at-most-two}, we conclude
that $J$ is either bipartite or \NB.
\end{proof}

In order to deduce a characterization of \pfaff\ \NB\ graphs in terms of forbidden conformal minors
(from Theorem~\ref{thm:near-bipartite-pfaffian-forbidden-S-minors}), it remains to figure out which
{\KFdec}s of $\{ K_{3,3}, \Gamma_1, \Gamma_2 \}$ are \NB;
those (along with $K_{3,3}$ itself) will be the forbidden {\comi}s; the rest may be omitted.
(Furthermore,
if any such graph contains a smaller such graph as a \comi, we may omit the former.)
To this end, we will find the following lemma useful.

\begin{lem}
\label{lem:near-bipartite-K4-decoration}
Let $R:=\{\alpha,\beta\}$ denote a removable doubleton of a cubic \NB\ graph~$G$,
let $X \subset V(G)$ such that $G[X]$ is isomorphic to~$C_3$, and let $C:=\partial(X)$.
Then $C$ is a separating $3$-cut of~$G$, and
precisely one of $\alpha$ and $\beta$, say~$\alpha$, is an edge of~$G[X]$.
Furthermore, let~$\alpha'$ denote the unique edge in $\partial(X)$ that is not adjacent with~$\alpha$,
and let $J$ denote the \mbox{$C$-contraction} \mbox{$G/X \rightarrow x$}.
Then $J$ is cubic, and precisely one of the following holds:
\begin{enumerate}[(i)]
\item either $\alpha' = \beta$ and $J$ is bipartite, or
\item otherwise $J$ is nonbipartite and $J-\alpha'-\beta$ is a bipartite matchable graph;
furthermore, if $J$ is $3$-connected then $J$ is \NB\ with $R':=\{\alpha',\beta\}$ as one of its removable doubletons.
\end{enumerate}
\end{lem}
\begin{proof}
Since $G$ is cubic, and $G[X]$ is isomorphic to $C_3$, it follows that $C$ is a $3$-cut.
By Proposition~\ref{prp:cubic-mcg-3cuts-are-separating-cuts}, $C$ is a separating cut of~$G$.

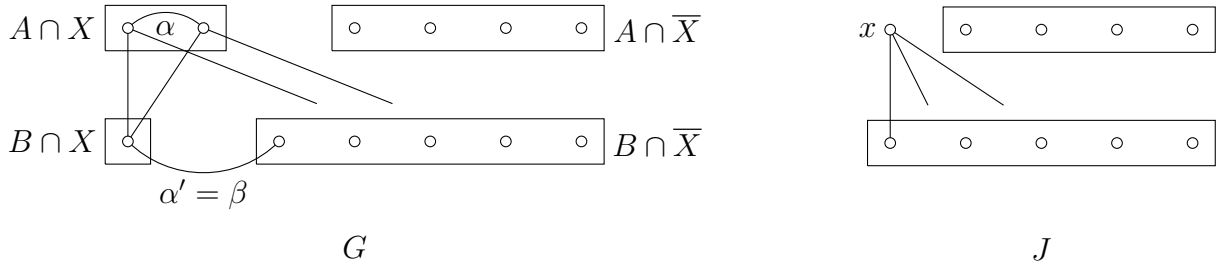
\begin{figure}[!htb]
\centering
\begin{tikzpicture}
\draw (-1,0)node[nodelabel]{$B \cap X$};
\draw (-1,1.5)node[nodelabel]{$A \cap X$};
\draw (7,0)node[nodelabel]{$B \cap \overline{X}$};
\draw (7,1.5)node[nodelabel]{$A \cap \overline{X}$};

\draw (0,1.5) -- (2.5,0.5);
\draw (1,1.5) -- (3.5,0.5);
\draw (0,0) -- (0,1.5);
\draw (0,0) -- (1,1.5);
\draw (0,1.5) to [out=45,in=135] (1,1.5);
\draw (0.5,1.5)node[nodelabel]{$\alpha$};

\draw (0,0) to [out=315,in=225] (2,0);
\draw (1,-0.7)node[nodelabel]{$\alpha' = \beta$};

\draw (-0.3,-0.3) -- (0.3,-0.3) -- (0.3,0.3) -- (-0.3,0.3) -- (-0.3,-0.3);
\draw (1.7,-0.3) -- (6.3,-0.3) -- (6.3,0.3) -- (1.7,0.3) -- (1.7,-0.3);
\draw (-0.3,1.2) -- (1.3,1.2) -- (1.3,1.8) -- (-0.3,1.8) -- (-0.3,1.2);
\draw (2.7,1.2) -- (6.3,1.2) -- (6.3,1.8) -- (2.7,1.8) -- (2.7,1.2);

\draw (0,0)node{};
\draw (2,0)node{};
\draw (3,0)node{};
\draw (4,0)node{};
\draw (5,0)node{};
\draw (6,0)node{};

\draw (0,1.5)node{};
\draw (1,1.5)node{};
\draw (3,1.5)node{};
\draw (4,1.5)node{};
\draw (5,1.5)node{};
\draw (6,1.5)node{};

\draw (3,-1.4)node[nodelabel]{$G$};

\end{tikzpicture}
\hspace*{0.5in}
\begin{tikzpicture}

\draw (2,1.5) -- (2.5,0.5);
\draw (2,1.5) -- (3.5,0.5);

\draw (2,1.5) -- (2,0);

\draw (1.7,-0.3) -- (6.3,-0.3) -- (6.3,0.3) -- (1.7,0.3) -- (1.7,-0.3);
\draw (2.7,1.2) -- (6.3,1.2) -- (6.3,1.8) -- (2.7,1.8) -- (2.7,1.2);

\draw (2,0)node{};
\draw (3,0)node{};
\draw (4,0)node{};
\draw (5,0)node{};
\draw (6,0)node{};

\draw (3,1.5)node{};
\draw (4,1.5)node{};
\draw (5,1.5)node{};
\draw (6,1.5)node{};

\draw (2,1.5)node{};
\draw (1.7,1.5)node[nodelabel]{$x$};

\draw (4,-1.4)node[nodelabel]{$J$};

\end{tikzpicture}
\vspace*{-0.2in}
\caption{Illustration for situation (i) in Lemma~\ref{lem:near-bipartite-K4-decoration}}
\label{fig:near-bipartite-K4-decoration-of-bipartite-graph}
\end{figure}

Let $A$ and $B$ denote the color classes of the bipartite graph~$G-R$, and adjust notation so that $\alpha$
has both ends in~$A$; consequently, $\beta$ has both ends in $B$. Since $G[X]$ is isomorphic to~$C_3$,
it follows immediately that precisely one of $\alpha$ and $\beta$ is an edge of~$G[X]$. Adjust notation so that
$\alpha$ is an edge of~$G[X]$. Two of the edges of~$C$ are adjacent with $\alpha$;
let $\alpha'$ denote the third edge that is not adjacent with~$\alpha$.
The reader may easily verify that if $\alpha' = \beta$ then $J:=G/X$ is bipartite;
see Figure~\ref{fig:near-bipartite-K4-decoration-of-bipartite-graph}.

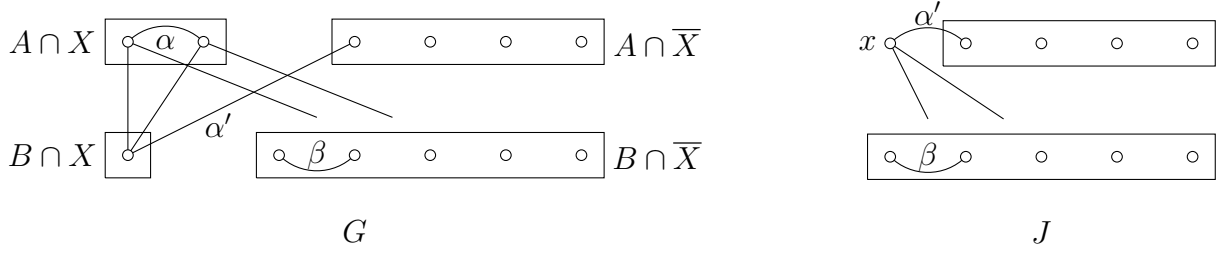
\begin{figure}[!htb]
\centering
\begin{tikzpicture}

\draw (-1,0)node[nodelabel]{$B \cap X$};
\draw (-1,1.5)node[nodelabel]{$A \cap X$};
\draw (7,0)node[nodelabel]{$B \cap \overline{X}$};
\draw (7,1.5)node[nodelabel]{$A \cap \overline{X}$};

\draw (2,0) to [out=315,in=225] (3,0);
\draw (2.5,0)node[nodelabel]{$\beta$};
\draw (0,0) -- (3,1.5);
\draw (1.2,0.4)node[nodelabel]{$\alpha'$};

\draw (0,1.5) -- (2.5,0.5);
\draw (1,1.5) -- (3.5,0.5);
\draw (0,0) -- (0,1.5);
\draw (0,0) -- (1,1.5);
\draw (0,1.5) to [out=45,in=135] (1,1.5);
\draw (0.5,1.5)node[nodelabel]{$\alpha$};

\draw (-0.3,-0.3) -- (0.3,-0.3) -- (0.3,0.3) -- (-0.3,0.3) -- (-0.3,-0.3);
\draw (1.7,-0.3) -- (6.3,-0.3) -- (6.3,0.3) -- (1.7,0.3) -- (1.7,-0.3);
\draw (-0.3,1.2) -- (1.3,1.2) -- (1.3,1.8) -- (-0.3,1.8) -- (-0.3,1.2);
\draw (2.7,1.2) -- (6.3,1.2) -- (6.3,1.8) -- (2.7,1.8) -- (2.7,1.2);

\draw (0,0)node{};
\draw (2,0)node{};
\draw (3,0)node{};
\draw (4,0)node{};
\draw (5,0)node{};
\draw (6,0)node{};

\draw (0,1.5)node{};
\draw (1,1.5)node{};
\draw (3,1.5)node{};
\draw (4,1.5)node{};
\draw (5,1.5)node{};
\draw (6,1.5)node{};

\draw (3,-1)node[nodelabel]{$G$};

\end{tikzpicture}
\hspace*{0.5in}
\begin{tikzpicture}

\draw (2,0) to [out=315,in=225] (3,0);
\draw (2.5,0)node[nodelabel]{$\beta$};
\draw (2,1.5) to [out=45,in=135] (3,1.5);
\draw (2.5,1.9)node[nodelabel]{$\alpha'$};

\draw (2,1.5) -- (2.5,0.5);
\draw (2,1.5) -- (3.5,0.5);

\draw (1.7,-0.3) -- (6.3,-0.3) -- (6.3,0.3) -- (1.7,0.3) -- (1.7,-0.3);
\draw (2.7,1.2) -- (6.3,1.2) -- (6.3,1.8) -- (2.7,1.8) -- (2.7,1.2);

\draw (2,0)node{};
\draw (3,0)node{};
\draw (4,0)node{};
\draw (5,0)node{};
\draw (6,0)node{};

\draw (3,1.5)node{};
\draw (4,1.5)node{};
\draw (5,1.5)node{};
\draw (6,1.5)node{};

\draw (2,1.5)node{};
\draw (1.7,1.5)node[nodelabel]{$x$};

\draw (4,-1)node[nodelabel]{$J$};

\end{tikzpicture}
\vspace*{-0.2in}
\caption{Illustration for situation (ii) in Lemma~\ref{lem:near-bipartite-K4-decoration}}
\label{fig:near-bipartite-K4-decoration-of-near-bipartite-graph}
\end{figure}

Now suppose that $\alpha' \neq \beta$;
see Figure~\ref{fig:near-bipartite-K4-decoration-of-near-bipartite-graph}.
Observe that $J - \alpha' - \beta$ is bipartite with color classes $A':=(A \cap \overline{X}) \cup \{x\}$
and $B':=B \cap \overline{X}$, and that $\alpha'$ has both ends in $A'$ whereas $\beta$ has both ends in~$B'$.
Since $J$ is \mc, it follows immediately that $J-\alpha'-\beta$ is matchable.
Finally, if $J$ is $3$-connected, we invoke
Proposition~\ref{prp:3conn-cubic-mcg-bip-index-at-most-two} to conclude
that $J$ is \NB\ with $R':=\{\alpha',\beta\}$ as one of its removable doubletons.
\end{proof}

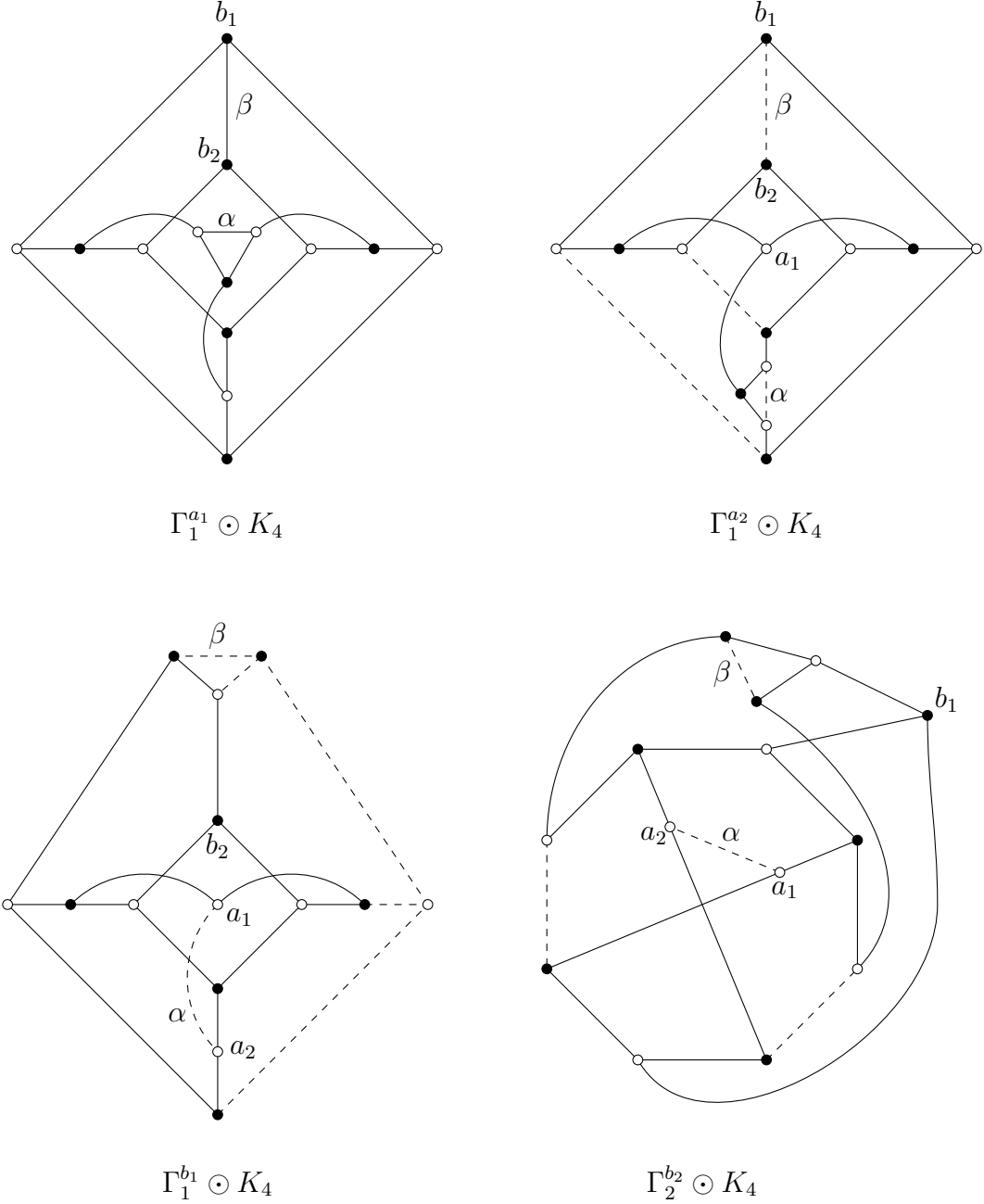
\begin{figure}[!htb]
\centering
\begin{tikzpicture}[scale=1.2]

\draw (150:0.4) to [out=135,in=45] (180:1.75);
\draw (30:0.4) to [out=45,in=135] (0:1.75);
\draw (270:0.4) to [out=225,in=135] (270:1.75);

\draw (30:0.4) -- (150:0.4) -- (270:0.4) -- (30:0.4);

\draw (90:0.35)node[nodelabel]{$\alpha$};

\draw (30:0.4)node{};
\draw (150:0.4)node{};
\draw (270:0.4)node[fill=black]{};

\draw (90:1) -- (180:1) -- (270:1) -- (0:1) -- (90:1);
\draw (90:2.5) -- (180:2.5) -- (270:2.5) -- (0:2.5) -- (90:2.5);

\draw (90:1) -- (90:2.5);
\draw (83:1.7)node[nodelabel]{$\beta$};
\draw (180:1) -- (180:1.75) -- (180:2.5);
\draw (270:1) -- (270:1.75) -- (270:2.5);
\draw (0:1) -- (0:1.75) -- (0:2.5);

\draw (90:1)node[fill=black]{};
\draw (100:1.2)node[nodelabel]{$b_2$};
\draw (180:1)node{};
\draw (270:1)node[fill=black]{};
\draw (0:1)node{};

\draw (180:1.75)node[fill=black]{};
\draw (270:1.75)node{};
\draw (0:1.75)node[fill=black]{};

\draw (90:2.5)node[fill=black]{};
\draw (90:2.8)node[nodelabel]{$b_1$};
\draw (180:2.5)node{};
\draw (270:2.5)node[fill=black]{};
\draw (0:2.5)node{};

\draw (270:3.3)node[nodelabel]{$\Gamma_1^{a_1} \odot K_4$};

\end{tikzpicture}
\hspace*{0.5in}
\begin{tikzpicture}[scale=1.2]

\draw (270:1.4) -- (260:1.75) -- (270:2.1);

\draw (0:0) to [out=135,in=45] (180:1.75);
\draw (0:0) to [out=45, in=135] (0:1.75);
\draw (0:0) to [out=225,in=135] (260:1.75);

\draw[dashed] (180:1) -- (270:1);
\draw (90:1) -- (180:1);
\draw (270:1) -- (0:1) -- (90:1);
\draw [dashed] (180:2.5) -- (270:2.5);
\draw (90:2.5) -- (180:2.5);
\draw (270:2.5) -- (0:2.5) -- (90:2.5);

\draw[dashed] (90:1) -- (90:2.5);
\draw (83:1.7)node[nodelabel]{$\beta$};
\draw (180:1) -- (180:1.75) -- (180:2.5);
\draw (270:1) -- (270:1.4);
\draw (270:2.1) -- (270:2.5);
\draw[dashed] (270:1.4) -- (270:2.1);
\draw (0:1) -- (0:1.75) -- (0:2.5);

\draw (0:0)node{};
\draw (330:0.3)node[nodelabel]{$a_1$};

\draw (90:1)node[fill=black]{};
\draw (90:0.7)node[nodelabel]{$b_2$};
\draw (180:1)node{};
\draw (270:1)node[fill=black]{};
\draw (0:1)node{};

\draw (180:1.75)node[fill=black]{};
\draw (0:1.75)node[fill=black]{};

\draw (90:2.5)node[fill=black]{};
\draw (90:2.8)node[nodelabel]{$b_1$};
\draw (180:2.5)node{};
\draw (270:2.5)node[fill=black]{};
\draw (0:2.5)node{};

\draw (270:1.4)node{};
\draw (270:2.1)node{};
\draw (260:1.75)node[fill=black]{};
\draw (275:1.75)node[nodelabel]{$\alpha$};

\draw (270:3.3)node[nodelabel]{$\Gamma_1^{a_2} \odot K_4$};

\end{tikzpicture}

\begin{tikzpicture}[scale=1.2]

\draw (90:2.5) -- (100:3);
\draw[dashed] (80:3) -- (90:2.5);

\draw (0:0) to [out=135,in=45] (180:1.75);
\draw (0:0) to [out=45, in=135] (0:1.75);
\draw[dashed] (0:0) to [out=225,in=135] (270:1.75);
\draw (250:1.4)node[nodelabel]{$\alpha$};

\draw (90:1) -- (180:1) -- (270:1) -- (0:1) -- (90:1);
\draw (180:2.5) -- (270:2.5);
\draw[dashed] (270:2.5) -- (0:2.5) -- (80:3) -- (100:3);
\draw (100:3) -- (180:2.5);

\draw (90:1) -- (90:2.5);
\draw (180:1) -- (180:1.75) -- (180:2.5);
\draw (270:1) -- (270:1.75) -- (270:2.5);
\draw (0:1) -- (0:1.75);
\draw[dashed] (0:1.75) -- (0:2.5);

\draw (0:0)node{};
\draw (330:0.3)node[nodelabel]{$a_1$};

\draw (90:1)node[fill=black]{};
\draw (90:0.7)node[nodelabel]{$b_2$};
\draw (180:1)node{};
\draw (270:1)node[fill=black]{};
\draw (0:1)node{};

\draw (180:1.75)node[fill=black]{};
\draw (270:1.75)node{};
\draw (280:1.75)node[nodelabel]{$a_2$};
\draw (0:1.75)node[fill=black]{};

\draw (90:2.5)node{};
\draw (80:3)node[fill=black]{};
\draw (100:3)node[fill=black]{};
\draw (90:3.2)node[nodelabel]{$\beta$};

\draw (180:2.5)node{};
\draw (270:2.5)node[fill=black]{};
\draw (0:2.5)node{};

\draw (270:3.3)node[nodelabel]{$\Gamma_1^{b_1} \odot K_4$};

\end{tikzpicture}
\hspace*{0.5in}
\begin{tikzpicture}[scale=1.2]

\draw[dashed] (85:3.2) -- (75:2.5);
\draw (85:2.75)node[nodelabel]{$\beta$};
\draw (75:2.5) -- (65:3.2) -- (85:3.2);
\draw (40:3.5) -- (65:3.2);

\draw (67.5:2) -- (40:3.5);
\draw (247.5:2) to [out=300,in=270] (0:2.8) to [out=90,in=270] (40:3.5);
\draw (40:3.5)node[fill=black]{};
\draw (40:3.8)node[nodelabel]{$b_1$};

\draw (157.5:2) to [out=90,in=180] (85:3.2);
\draw (337.5:2) to [out=45,in=330] (75:2.5);

\draw (85:3.2)node[fill=black]{};
\draw (75:2.5)node[fill=black]{};
\draw (65:3.2)node{};

\draw (22.5:2) -- (202.5:2);
\draw (112.5:2) -- (292.5:2);
\draw[dashed] (22.5:1) -- (112.5:1);
\draw (67.5:0.9)node[nodelabel]{$\alpha$};

\draw (22.5:1)node{};
\draw (10:1)node[nodelabel]{$a_1$};
\draw (112.5:1)node{};
\draw (125:1)node[nodelabel]{$a_2$};

\draw (337.5:2) -- (22.5:2) -- (67.5:2) -- (112.5:2) -- (157.5:2);
\draw[dashed] (157.5:2) -- (202.5:2);
\draw (202.5:2) -- (247.5:2) -- (292.5:2);
\draw[dashed] (292.5:2) -- (337.5:2);

\draw (22.5:2)node[fill=black]{};
\draw (67.5:2)node{};
\draw (112.5:2)node[fill=black]{};
\draw (157.5:2)node{};
\draw (202.5:2)node[fill=black]{};
\draw (247.5:2)node{};
\draw (292.5:2)node[fill=black]{};
\draw (337.5:2)node{};

\draw (270:3.3)node[nodelabel]{$\Gamma_2^{b_2} \odot K_4$};

\end{tikzpicture}
\vspace*{-0.4in}
\caption{Near-bipartite {\KFdec}s of Cubeplex and Twinplex of order $14$}
\label{fig:K4-decorations-of-Cubeplex-and-Twinplex}
\end{figure}

The following is a useful consequence of Lemma~\ref{lem:near-bipartite-K4-decoration} which will help us in
easily identifying the \NB\ {\KFdec}s of $\{K_{3,3}, \Gamma_1, \Gamma_2\}$.

\begin{cor}
\label{cor:near-bipartite-K4-decoration}
Let $J$ denote a $3$-connected cubic \NB\ graph, let $v \in V(J)$, and let $G:= J^v \odot K_4$.
Then $G$ is \NB\ if and only if $v \in V(R)$ for some removable doubleton~$R$ of~$J$.
Furthermore, if $J$ has a unique removable doubleton, say~$R:=\{\alpha,\beta\}$, and if
$v$ is an end of $\alpha$, then $G$ has a unique removable doubleton containing $\beta$.
\qed
\end{cor}

We let $\mathcal{N}$ denote the set comprising five \NB\ graphs:
$K_{3,3} \odot K_4$ (shown in Figure~\ref{fig:solid-K4-decorations-of-K33}),
$K_{3,3}^{a_2,b_2} \odot K_4$ (shown in Figure~\ref{fig:smallest-nonsolid-K4-decoration-of-K33}),
$\Gamma_1$ and $\Gamma_2$ (shown in Figure~\ref{fig:Cubeplex-and-Twinplex}), and
$\Gamma_1^{a_1} \odot K_4$ (shown in Figure~\ref{fig:K4-decorations-of-Cubeplex-and-Twinplex}).

\begin{prp}
\label{prp:near-bipartite-K4-decorations-of-K33-Cubeplex-Twinplex}
Each \NB\ graph in $K_4(\{K_{3,3}, \Gamma_1, \Gamma_2\})$ either belongs to $\mathcal{N}$
or is otherwise \KTTB.
\end{prp}
\begin{proof}
We omit proof of the fact that each graph in $\mathcal{N}$ is \KTTF. It may be verified computationally,
or otherwise.

\vfill

Let $J \in \{K_{3,3}, \Gamma_1, \Gamma_2\}$ and let $G:=J^T \odot K_4$, for some $T \subseteq V(J)$,
so that $G$ is \NB.
In particular, ${\sf bi}(G)=2$; this immediately implies that $|T| \leq 2$. If $T = \emptyset$ then
the desired conclusion holds.
Henceforth, we consider $|T| \geq 1$.

First suppose that $J = K_{3,3}$. If $|T|=1$ then $G = K_{3,3} \odot K_4$;
see Figure~\ref{fig:solid-K4-decorations-of-K33}.
The reader may easily verify that $K_{3,3} \odot K_4$ has three removable doubletons --- each comprising
one edge from the unique $C_3$, and another from the corresponding cut $\partial(V(C_3))$. Consequently,
by Corollary~\ref{cor:near-bipartite-K4-decoration},
if $|T|=2$ then $G$ is obtained by splicing $K_4$ with $K_{3,3} \odot K_4$ at one of the three vertices
which is at distance one from the unique $C_3$.
It is easily verified that these three vertices belong to the same orbit of ${\sf Aut}(K_{3,3} \odot K_4)$,
and that $G = K_{3,3}^{a_2,b_2} \odot K_4$
(as shown in Figure~\ref{fig:smallest-nonsolid-K4-decoration-of-K33}).
In each case, $G \in \mathcal{N}$.

Now suppose that $J \in \{ \Gamma_1, \Gamma_2\}$. Since each of $\Gamma_1$ and $\Gamma_2$
has a unique removable doubleton (as indicated in Figure~\ref{fig:Cubeplex-and-Twinplex}),
we invoke Corollary~\ref{cor:near-bipartite-K4-decoration} ($|T|$ times)
to infer that $T \subseteq \{a_1,a_2,b_1,b_2\}$
and that $G$ has a unique removable doubleton, say~$R$;
furthermore, if $|T|=2$ then one of its vertices belongs to $\{a_1,a_2\}$ and the other belongs to $\{b_1,b_2\}$.

If $|T|=1$, it follows from Proposition~\ref{prp:automorphisms-of-Cubeplex-and-Twinplex}
that $G$ is one of the four graphs shown in Figure~\ref{fig:K4-decorations-of-Cubeplex-and-Twinplex}.
Observe that each of these graphs --- except for $\Gamma_1^{a_1} \odot K_4$ ---
has a conformal subgraph~$H$ that is a bi-subdivision of $K_{3,3}$
such that no vertex in $V(R)$ is a cubic vertex of~$H$. (To see such an $H$, delete the dashed edges
in the figure, and delete the resulting isolated vertices --- if any.) The desired conclusion holds.

If $|T|=2$, then $G$ may be easily viewed as a \KFdec\ of one of the three \KTTB\ graphs
shown in Figure~\ref{fig:K4-decorations-of-Cubeplex-and-Twinplex}. Using the observation stated above
(regarding the subgraph~$H$) and Corollary~\ref{cor:splicing-at-good-vertex},
we deduce that $G$ is \KTTB.

This completes the proof of Propostion~\ref{prp:near-bipartite-K4-decorations-of-K33-Cubeplex-Twinplex}.
\end{proof}

We need one last fact that is easily proved using Menger's Theorem.

\begin{lem}
\label{lem:K4-decoration-3-connected-cubic-graph}
Each \KFdec\ of a $3$-connected cubic graph is $3$-connected and cubic. \qed
\end{lem}

We are now ready to deduce (from Theorem~\ref{thm:near-bipartite-pfaffian-forbidden-S-minors})
the characterization of \pfaff\ \NB\ graphs
in terms of forbidden conformal minors due to Fischer and Little~\cite{fili01}.

\begin{thm}
\label{thm:near-bipartite-pfaffian-forbidden-conformal-minors}
A near-bipartite graph~$G$ is \pfaff\ if and only if $G$ is $(\{K_{3,3}\} \cup \mathcal{N}$)-free.
\end{thm}
\begin{proof}
It follows immediately from Theorem~\ref{thm:near-bipartite-pfaffian-forbidden-S-minors} and
our Main Theorem (\ref{thm:conformal-minors-vs-S-minors-cubic-long-version})
that a \NB~$G$ is \pfaff\ if and only if $G$ is $K_4(\{K_{3,3}, \Gamma_1, \Gamma_2\})$-free.
By Lemma~\ref{lem:K4-decoration-3-connected-cubic-graph}, each of the graphs
in $K_4(\{K_{3,3}, \Gamma_1, \Gamma_2\})$ is $3$-connected and cubic.
By invoking Corollary~\ref{cor:3-connected-cubic-conformal-minors-near-bipartite}
and Proposition~\ref{prp:near-bipartite-K4-decorations-of-K33-Cubeplex-Twinplex},
we conclude that a \NB\ graph~$G$ is \pfaff\ if and only if $G$ is $\{K_{3,3}\} \cup \mathcal{N}$-free.
\end{proof}

In their paper, Fischer and Little \cite{fili01} included seven graphs in their list of forbidden conformal minors.
However, one of them --- namely, $\Gamma_{1,1}$ in their paper which is the same
as~$\Gamma_1^{b_1} \odot K_4$
shown in Figure~\ref{fig:K4-decorations-of-Cubeplex-and-Twinplex} --- is redundant (since it is \KTTB).
One may verify using computations or otherwise that the list of six forbidden conformal minors
(in Theorem~\ref{thm:near-bipartite-pfaffian-forbidden-conformal-minors}) is best possible
since none of them contains another as a conformal minor; we omit proof of this fact.

\bigskip
\noindent
{\bf Acknowledgements}: The first author would like to thank Zolt{\'a}n Szigeti for the many
stimulating discussions
during the latter's visit to the University of Waterloo (in 2014).
During one such discussion, Zolt{\'a}n mentioned the weaker
version of Theorem~\ref{thm:claw-theta-K4-comi} (proved in \cite{abss11}),
and suggested that one should be able to prove
the stronger version easily.

\bibliographystyle{alpha}
\bibliography{clm}

\end{document}